\documentclass[letterpaper]{amsart}
\pdfoutput=1
\newtheorem{theorem}{Theorem}[section]

\newtheorem{proposition}[theorem]{Proposition}

\theoremstyle{definition}
\newtheorem{definition}[theorem]{Definition}

\theoremstyle{remark}
\newtheorem{remark}[theorem]{Remark}
\usepackage{hyperref}
\usepackage[top=1.0in, left=1.0in, right=1.0in]{geometry}
\usepackage[all,cmtip]{xy}
\usepackage{xypic}
\usepackage{pgf}
\usepackage{graphicx}
\usepackage{color}
\usepackage{amsmath}
\usepackage{amsfonts}
\usepackage{euler}
\usepackage{amssymb}
\usepackage{epsfig}
\usepackage{rotating}
\usepackage{graphics}
\newcommand{\be}{\begin{equation}}

\newcommand{\ee}{\end{equation}}

\newcommand{\dz}{\wedge}

\newcommand{\ba}{\begin{array}}

\newcommand{\ea}{\end{array}}

\newcommand{\beq}{\begin{eqnarray}}

\newcommand{\eeq}{\end{eqnarray}}

\newtheorem{lm}{lemma}

\newtheorem{thee}{theorem}

\newtheorem{proo}{proposition}

\newtheorem{co}{corollary}

\newtheorem{rem}{remark}

\newtheorem{deff}{definition}

\newcommand{\bd}{\begin{deff}}

\newcommand{\ed}{\end{deff}}

\newcommand{\bl}{\begin{lm}}

\newcommand{\el}{\end{lm}}

\newcommand{\bp}{\begin{proo}}

\newcommand{\ep}{\end{proo}}

\newcommand{\bt}{\begin{thee}}

\newcommand{\et}{\end{thee}}

\newcommand{\bc}{\begin{co}}

\newcommand{\ec}{\end{co}}

\newcommand{\brm}{\begin{rem}}

\newcommand{\erm}{\end{rem}}

\newcommand{\der}{{\rm d}}

\usepackage{t1enc}
\def\frak{\mathfrak}

\newcommand{\newc}{\newcommand}

\renewcommand{\exp}{\operatorname{exp}}
\newcommand{\id}{\operatorname{id}}

\let\ccdot\cdot
\def\cdot{\hbox to 2.5pt{\hss$\ccdot$\hss}}

\newc{\aR}{\mbox{\boldmath{$ R$}}}
\newc{\aS}{\mbox{\boldmath{$ S$}}}
\newc{\aT}{\mbox{\boldmath{$ T$}}}
\newc{\aW}{\mbox{\boldmath{$ W$}}}

\newc{\aK}{\mbox{\boldmath{$ K$}}}
\newc{\aL}{\mbox{\boldmath{$ L$}}}

\newcommand{\bbC}{\mathbb{C}}

\usepackage{amssymb}
\usepackage{amscd}

\newcommand{\hook}{\raisebox{-0.35ex}{\makebox[0.6em][r]
{\scriptsize $-$}}\hspace{-0.15em}\raisebox{0.25ex}{\makebox[0.4em][l]{\tiny
 $|$}}}

\newcommand{\bma}{\begin{pmatrix}}
\newcommand{\ema}{\end{pmatrix}}

\newc{\obstrn}[2]{B^{#1}_{#2}}

\newcommand{\rpl}                         
{\mbox{$
\begin{picture}(12.7,8)(-.5,-1)
\put(0,0.2){$+$}
\put(4.2,2.8){\oval(8,8)[r]}
\end{picture}$}}

\newcommand{\lpl}                         
{\mbox{$
\begin{picture}(12.7,8)(-.5,-1)
\put(2,0.2){$+$}
\put(6.2,2.8){\oval(8,8)[l]}
\end{picture}$}}

\usepackage{ifthen}

\newc{\tensor}[1]{#1}
\newc{\Mvariable}[1]{\mbox{#1}}
\newc{\down}[1]{{}_{#1}}
\newc{\up}[1]{{}^{#1}}

\newc{\JulyStrut}{\rule{0mm}{6mm}}
\newc{\midtenPan}{\mbox{\sf S}}
\newc{\midten}{\mbox{\sf T}}
\newc{\midtenEi}{\mbox{\sf U}}
\newc{\ATen}{\mbox{\sf E}}
\newc{\BTen}{\mbox{\sf F}}
\newc{\CTen}{\mbox{\sf G}}

\def\sideremark#1{\ifvmode\leavevmode\fi\vadjust{\vbox to0pt{\vss
 \hbox to 0pt{\hskip\hsize\hskip1em
 \vbox{\hsize3cm\tiny\raggedright\pretolerance10000
 \noindent #1\hfill}\hss}\vbox to8pt{\vfil}\vss}}}%

\newcommand{\bgs}{{\textstyle \bigodot}}
\newcommand{\bgt}{{\textstyle \bigotimes}}
\newcommand{\Span}{\mathrm{Span}}

\numberwithin{equation}{section}

\newcounter{romenumi}
\newcommand{\labelromenumi}{(\roman{romenumi})}

\begin{document}
\title{Aerobatics of flying saucers}
\vskip 1.truecm
\author{Michael Eastwood} \address{School of Mathematical Sciences,
University of Adelaide, SA 5005, Australia}
\email{meastwoo@member.ams.org}
\author{Pawe\l~ Nurowski} \address{Center for Theoretical Physics ,
PAS, Al. Lotnik\'ow 32/46, 02-668 Warszawa, Poland}
\email{nurowski@cft.edu.pl}
\thanks{This work was supported by the Simons Foundation grant 346300 and the
Polish Government MNiSW 2015--2019 matching fund. It was written whilst
the first author was visiting the Banach Centre at IMPAN in Warsaw for
the Simons Semester `Symmetry and Geometric Structures' and during
another visit to Warsaw supported by the Polish National Science Centre
(NCN) via the POLONEZ grant 2016/23/P/ST1/04148, which received funding from
the European Union's Horizon 2020 research and innovation programme under the
Marie Sk{\l}odowska-Curie grant agreement No.~665778.}

\begin{abstract}
Starting from the observation that a flying saucer is a nonholonomic mechanical system whose
$5$-dimensional configuration space is a contact manifold, we show how to enrich this space
with a number of geometric structures by imposing further nonlinear restrictions on 
the saucer's velocity. These restrictions define certain `man{\oe}uvres' of the saucer, which 
we call `attacking,' `landing,' or `$G_2$~mode' man{\oe}uvres,
and which equip its configuration space with three kinds of flat parabolic geometry in 
five dimensions. The attacking man{\oe}uvre corresponds to the flat Legendrean 
contact structure, the landing man{\oe}uvre corresponds to the flat hypersurface type 
CR structure with Levi form of signature $(1,1)$, and the most complicated $G_2$ 
man{\oe}uvre corresponds to the contact Engel structure \cite{engel,katja} with split 
real form of the exceptional Lie group $G_2$ as its symmetries.
A celebrated double fibration relating the two nonequivalent flat $5$-dimensional
parabolic $G_2$ geometries is used to construct a `$G_2$~joystick,' consisting of 
two balls of radii in ratio $1\!:\!3$ that transforms the difficult 
$G_2$ man{\oe}uvre into the pilot's action of rolling one of joystick's balls on the other 
without slipping nor twisting.      
\end{abstract}
\maketitle
\vspace{-1truecm}
\tableofcontents
\newcommand{\bbS}{\mathbb{S}}
\newcommand{\bbR}{\mathbb{R}}
\newcommand{\sog}{\mathbf{SO}}
\newcommand{\glg}{\mathbf{GL}}
\newcommand{\slg}{\mathbf{SL}}
\newcommand{\og}{\mathbf{O}}
\newcommand{\soa}{\frak{so}}
\newcommand{\gla}{\frak{gl}}
\newcommand{\sla}{\frak{sl}}
\newcommand{\sua}{\frak{su}}
\newcommand{\dr}{\mathrm{d}}
\newcommand{\sug}{\mathbf{SU}}
\newcommand{\cspg}{\mathbf{CSp}}
\newcommand{\gat}{\tilde{\gamma}}
\newcommand{\Gat}{\tilde{\Gamma}}
\newcommand{\thet}{\tilde{\theta}}
\newcommand{\Thet}{\tilde{T}}
\newcommand{\rt}{\tilde{r}}
\newcommand{\st}{\sqrt{3}}
\newcommand{\kat}{\tilde{\kappa}}
\newcommand{\kz}{{K^{{~}^{\hskip-3.1mm\circ}}}}
\newcommand{\bv}{{\bf v}}
\newcommand{\di}{{\rm div}}
\newcommand{\curl}{{\rm curl}}
\newcommand{\cs}{(M,{\rm T}^{1,0})}
\newcommand{\tn}{{\mathcal N}}
\newcommand{\ten}{{\Upsilon}}
\section{Introduction}\label{intr}
\subsection{What is a flying saucer?}
Let us start with the Wikipedia definition \cite{wiki}: \\\\
\emph{
A flying saucer (also referred to as a flying disc) is a type of described flying craft with a disc or saucer-shaped body, commonly used generically to refer to any anomalous flying object.}\\\\ 
In this short note we show that \emph{a flying saucer}, i.e.~a disc-shaped body that can move in
$3$-space according to natural rules, is an example of an interesting
 \emph{nonholonomic system}. In particular, \emph{we define certain classes of aerobatics of flying saucers}, which equip their configuration spaces with structures of a number of \emph{flat parabolic geometries in dimension five}. This will include flat \emph{Legendrean contact geometry} which has $\sla(4,\bbR)$ as its Lie algebra of symmetries, as well as a flat ${\bf G}_2$  \emph{contact} \emph{geometry} with symmetry algebra isomorphic to the split real form of the exceptional Lie algebra $\mathfrak{g}_2$.

This paper is concerned with \emph{flat models} for a saucer's man{\oe}uvres. It has a companion 
\cite{dynamics} explaining how to obtain non-flat flying saucer configuration spaces from 
simple geometric data in the 3-space where the saucers navigate.   
\subsection{Configuration space of a flying saucer}
Consider $\bbR^3$ with the orthonormal vectors $(\vec{e}_x,\vec{e}_y,\vec{e}_z)$. 
To specify the position of a flying saucer we need to specify a point 
$\vec{r}=x\vec{e}_x+y\vec{e}_y+z\vec{e}_z$ in $\bbR^3$, and to chose a unit vector 
$\vec{n}$ in $\bbR^3$. The point $\vec{r}$ gives the position of saucer's center of mass 
in 3-space. The vector $\vec{n}$ gives the unit normal to the saucer's disk. 
Thus, the unit vector $\vec{n}$ determines the orientation of the saucer in 3-space. When saucer moves its \emph{center of mass may assume any position in space} $\bbR^3$, and \emph{the endpoint of the normal vector to its disk may assume any position on the unit sphere} $\bbS^2$. This shows that \emph{the configuration space of the flying saucer is} $${\mathcal C}=\bbR^3\times \bbS^2.$$ 
For the later convenience we recall that in the standard coordinates $(x,y,z)$ in $\bbR^3$, and $(\theta,\phi)$ on $\bbS^2$, the respective volume forms are:
$$vol_{\bbR^3}=\der x\dz\der y\dz \der z,\quad\quad vol_{\bbS^2}=sin\theta\der\theta\dz\der\phi.$$
Using a unit vector $\vec{n}$, the volume $vol_{\bbS^2}$ can be expressed in a 
coordinate-free way as:
$$vol_{\bbS^2}=\tfrac12 (\vec{n}\times \der{\vec n})\dz\der\vec{n}.$$
This formula can be checked by substituting in it $\vec{n}=cos\phi sin\theta\vec{e}_x+sin\phi sin\theta\vec{e}_y+cos\theta\vec{e}_z$, and comparing it with $sin\theta\der\theta\dz\der\phi.$

Note that we obviously have:
$$vol_{\bbS^2}\dz vol_{\bbR^3}\neq 0,~{\rm everywhere}~{\rm in}~{\mathcal C}=\bbR^3\times\bbS^2.$$

\subsection{Movement of a flying saucer}
A flying saucer, when it moves, describes a curve
$$\gamma(t)=(\vec{r}(t),\vec{n}(t))\subset{\mathcal C}$$
in the configuration space $\mathcal C$. In the following we will only consider \emph{piecewise smooth movements} of the saucer, corresponding to piecewise smooth curves $\gamma(t)$. 

Because of the split ${\mathcal C}=\bbR^3\times\bbS^2$ the curve $\gamma(t)=(\vec{r}(t),\vec{n}(t))$ defines two curves 
$\gamma_1(t)=\vec{r}(t)\in\bbR^3$ and $\gamma_2(t)=\vec{n}(t)\in\bbS^2$. 

Consider now the second curve, $\gamma_2(t)$, and two moments of time $t$ and $t+dt$. This corresponds to two unit vectors $\vec{n}(t)$ and $\vec{n}(t+dt)=\vec{n}(t)+\der \vec{n}(t)$ anchored at the origin in $\bbR^3$, corresponding to two orientations, at time $t$ and $t+\der t$ respectively, of the disk of the saucer in the 3-space. The vector $\vec{n}(t)$ is orthogonal to the change vector $\der\vec{n}(t)$ because  differentiating $\vec{n}^2(t)=1$ we get $\vec{n}(t)\der \vec{n}(t)=0$. Thus, if we consider a movement $\gamma(t)=(\vec{r}(t),\vec{n}(t))$ of a saucer during which $\vec{n}(t)$ changes, we have a triple of orthogonal vectors $(\vec{n}(t), \der\vec{n}(t), \vec{n}(t)\times \der\vec{n}(t))$ attached to every point of the curve $\gamma_2(t)$. The physical interpretation of the vector $\vec{n}(t)\times\der\vec{n}(t)$ is such that this vector is tangent to the line $l(t)$ in the plane of the saucer around which the saucer rotates when changing its orientation from the one described by $\vec{n}(t)$ to the one described by $\vec{n}(t+\der t)$.   

\subsection{Flying saucer movement is nonholonomic}
For the purpose of this article we assume that flying saucers \emph{cannot move in the direction perpendicular to their disk}. Thus our main assumption about the flying saucer moving along the curve $$\gamma(t)=(\vec{r}(t),\vec{n}(t))$$ is that at every moment of the movement we have:
\be
\dot{\vec{r}}(t) \perp \vec{n}(t).\label{nh}\ee
This gives a linear restriction on possible velocities of the flying saucer, placing its kinematics in the realm of \emph{nonholonomic} systems.
In terms of the geometry of the configuration space $\mathcal C$, the nonholonomic condition (\ref{nh}) defines a 4-dimensional distribution $\mathcal D$ in $\mathcal C$ which at each point $(\vec{r},\vec{n})\in\mathcal C$ is annihilated by the one form
\be\omega^0=\vec{n}\cdot\der\vec{r}.\label{con}\ee 
Here we used the fact that $\der\vec{r}(t)=\dot{\vec{r}}(t)\der t$. Of course, we also have $\der\vec{n}(t)=\dot{\vec{n}}(t)\der t$, etc. 

We end this section with a formal definition of the 4-distribution $\mathcal D$ of the admissible velocities:
$${\mathcal D}=\{~\Gamma({\rm T}{\mathcal C})\ni X~~|~~ X\hook(\vec{n}\cdot\der\vec{r}) =0\}={\omega^0}^\perp.$$

\subsection{Naked flying saucer is a contact five manifold} 
Calculating the maximal wedge product $\der\omega^0\dz\der\omega^0\dz\omega^0$ of forms obtained naturally from the annihilator $\omega^0$ of the velocity distribution $\mathcal D$ we get:
$$\begin{aligned}
\der \omega^0&\dz\der\omega^0\dz\omega^0=\\&(\der\vec{n}\dz\vec\der{r})\dz (\der\vec{n}\dz\vec\der{r})\dz(\vec{n}\der\vec{r})=
\der n_j\dz\der x_j\dz\der n_k\dz \der x_k\dz n_i\der x_i=\\&
-\epsilon_{ijk}n_i\der n_j\dz\der n_k\dz\der x\dz\der y\dz\der z=-(\vec{n}\times\der \vec{n})\dz\der\vec{n}\dz \der x\dz\der y\dz\der z=\\-2&vol_{\bbS^2}\dz vol_{\bbR^3},
\end{aligned}$$
and thus $\der\omega^0\dz\der\omega^0\dz\omega^0$ \emph{is never zero} on the configuration space $\mathcal C$ of the saucer. This means that the rank 4-distribution $\mathcal D$ of possible velocities of the saucer is a \emph{contact} distribution on $\mathcal C$. 

Let us summarize our considerations in the following proposition.
\begin{proposition}
The five-dimensional configuration space $\mathcal C=\bbR^3\times\bbS^2$ of a flying saucer is naturally a \emph{contact} manifold $({\mathcal C},{\mathcal D})$. In the natural coordinates $(\vec{r},\vec{n})$ in $\mathcal C$ the contact distribution $\mathcal D$ is given as the annihilator of a field of a 1-form $\omega^0=\vec{n}\der\vec{r}$, which is given on $\mathcal C$ up to a scale. As with every contact distribution, the distribution $\mathcal D$ is equipped with a family of nondegenerate two forms $\Omega=\der \omega^0$, which are given on the distribution also up to a scale. 
\end{proposition}

In the following, occasionally, we will use local coordinates in $\mathcal C$. In particular, we will consider an open subset ${\mathcal C}_>$ of $\mathcal C$ defined by:
$${\mathcal C}_>=\{~{\mathcal C}\ni (\vec{r},\vec{n}) ~~|~~\vec{n}\cdot \vec{e}_z>0~\}.$$ In $\mathcal C_>$ the points $\vec{n}$ are taken from the \emph{northern hemisphere} of $\bbS^2$. Thus, ${\mathcal C}_>$ is diffeomorphic to $\bbR^5$, $${\mathcal C}_>\equiv \bbR^5,$$ and as such may be parametrized by \emph{five real numbers} $(x,y,z,a,b)$. Here 
\be
x\vec{e}_x+y\vec{e}_y+z\vec{e}_z=\vec{r},\label{tr1}\ee 
and $(a,b)$ are determined by 
the requirement that 
$$z=ax+by~{\rm is}~{\rm the}~{\rm plane}~{\rm orthogonal}~{\rm to}~\vec{n}~{\rm and}~{\rm passing}~{\rm through}~{\rm the}~{origin}~{\rm in}~\bbR^3.$$ In particular we have:
\be
\vec{n}=\frac{\vec{N}}{||\vec{N}||},\quad{\rm and}\quad\vec{N}=-a \vec{e}_x-b\vec{e}_y+\vec{e}_z .\label{tr2}\ee
We can easily pass between $(\vec{r},\vec{n})\in {\mathcal C}_>$ and $(x,y,z,a,b)\in\bbR^5$ using (\ref{tr1}) and (\ref{tr2}). 

In our parametrization $(x,y,z,a,b)$ of ${\mathcal C}_>$ the one form $\omega^0$ reads:
$$\omega^0=\vec{n}\cdot\der\vec{r}=\frac{\der z - a \der x-b\der y}{||\vec{N}||}$$
and we can rescale it to obtain: 
\be\omega^0=\der z - a \der x-b\der y.\label{con1}\ee
Thus, the parametrization is the standard Darboux parametrization $(x_i,z,p_i)$ of contact forms, which in any odd dimension can be always locally written as 
$\omega=\der z+p_i\der x_i$. In our case $(x_1,x_2)=(x,y)$ and $(p_1,p_2)=(-a,-b)$. The advantage of using this parametrization is that, in particular we immediately see that:
$$\der\omega^0\dz\der\omega^0\dz\omega^0=2\der x\dz\der y\dz\der a\dz\der b\dz\der z\neq 0.$$

The distribution $\mathcal D$ is crucial to describe the process of controlling a flying saucer. If $(E_1,E_2,E_3,E_4)$ are vector fields that \emph{locally} span $\mathcal D$ and the curve $\gamma(t)\subset\mathcal C$ is the trajectory of the saucer, then the \emph{velocity} of the saucer $\dot{\gamma}(t)$ must be of the form
$$\dot{\gamma}(t)=u_1(t)E_1+u_2(t)E_2+u_3(t)E_3+u_4(t)E_4.$$
The functions $(u_1(t),u_2(t),u_3(t),u_4(t))$ are called the \emph{controls} of the saucer.  

In our parametrization of ${\mathcal C}_>$ we have:
$${\mathcal D}=\Span(E_1,E_2,E_3,E_4),$$
where
$$E_1=\partial_x+a\partial_z,\quad E_2=\partial_y+b\partial_z,\quad E_3=\partial_b,\quad E_4=\partial_a.$$

In the following sections we will use the directions $\vec{n}(t)$ and $\der\vec{n}(t)$, which at every moment of time are defined by the curve 
$\gamma(t)\subset{\mathcal C}$, to describe \emph{two different classes of movements} of a flying saucer. We will call these movements `\emph{attacking mode aerobatics}' and `\emph{landing mode aerobatics}', respectively. They will be described in Sections \ref{sa1}-\ref{sa2}.

If there are no further restrictions on the movement of the saucer its \emph{pilot can use the engines} to profile the control functions $(u_1(t),u_2(t),u_3(t),u_4(t))$ properly, to give a desired shape of his trajectory $\gamma_1(t)=\vec{r}(t)$ in 3-space $\bbR^3$. Since the distribution $\mathcal D$ is contact, he is assured by the Chow-Raszewski theorem, that by staying on paths which are always tangent to $\mathcal D$, he can move his saucer from any point in the configuration space $\mathcal C$ to any other point. 

To make the life of a pilot of a flying saucer more adventurous, we will now make further restrictions on the man{\oe}uvres. This will be \emph{aerobatic} man{\oe}uvres of the flying saucer, which will relate its velocity $\dot{\vec{r}}(t)$ in space to its locally orthogonal frame given by $(~\vec{n}(t),\der\vec{n}(t),\vec{n}(t)\times\der\vec{n}(t)~)$. We assume that the pilot of a saucer can observe the data given by $$c(t)=(\vec{r}(t),\vec{n}(t),\dot{\vec{r}}(t),\dot{\vec{n}}(t),\vec{n}(t)\times\dot{\vec{n}}(t))$$
on the instruments of the saucer, and that he can use his controls to impose appropriate relations between the components of the vector $c(t)$. We call this relations \emph{aerobatics}.

\section{Attacking mode aerobatics - 5-dimensional Legendrean contact structure}\label{sa1}  An \emph{attacking mode aerobatic man{\oe}uvre} consists in a movement of a flying saucer in such a way that the the speed $\dot{\vec{r}}(t)$ of the saucer in 3-space is, at every moment of time, parallel to the line in the plane of the saucer around which the saucer is momentarily rotating. More formally: in an attacking mode aerobatic man{\oe}uvre, at every moment $t$ of time, \emph{the speed} $\dot{\vec{r}}(t)$ \emph{of the saucer in 3-space is parallel to the line defined by the vector} $\vec{n}(t)\times\der\vec{n}(t)$. Recall that the vector $\vec{n}(t)\times\der\vec{n}(t)$ is always in the plane of the saucer, so this rule alone implies that $\dot{\vec{r}}\perp\vec{n}$. Thus the movement during the attacking mode aerobatic man{\oe}uvre automatically satisfies the nonholonomic constraint stating that the corresponding curve in the configuration space is tangent to $\mathcal D$.  

In physical terms this rule says that at the moment when the saucer spins around the line tangent to $\vec{n}(t)\times\der\vec{n}(t)$, its center of mass has velocity along this line. An example of such a movement is a movement of a rifle bullet, which spins around the axis determined by the tangent to its trajectory.

If the pilot of a saucer does only attacking mode aerobatics his configuration space $\mathcal C$ is equipped with a richer structure than just $({\mathcal C},{\mathcal D})$. 

\subsection{Conformal metric on the distribution} The \emph{attacking mode aerobatics rule} says that at every moment of time $t$ we have:
$$\der\vec{r}(t) ~||~ \Big(\vec{n}(t)\times\der\vec{n}(t)\Big).$$
This means that the cross product of the two vectors $\der\vec{r}$ and $\vec{n}\times\der\vec{n}$ is zero:
\be \der\vec{r}\times(\vec{n}\times\der\vec{n})~=~0.\label{metz1}\ee
An identity from the vector calculus then yields:
$$0=\der\vec{r}\times(\vec{n}\times\der\vec{n})=(\der\vec{r}\cdot\der\vec{n})\vec{n}-(\vec{n}\cdot\der\vec{r})\der\vec{n}.$$
Since on the distribution $\mathcal D$ the form $\omega^0=\vec{n}\cdot\der\vec{r}\equiv 0$, we get 
$$(\der\vec{r}\cdot\der\vec{n})\vec{n}=0,$$
which leads to the conclusion that the saucer performs the attacking mode aerobatic manouever if and only if its path $(\vec{r}(t),\vec{n}(t))$ in the configuration space satisfies the following conditions: 
\be
\dot{\vec{r}}~\cdot~\dot{\vec{n}}=0.\label{metz}\ee
Of course we need also to satisfy the tangency to the distribution condition, which means that 
$$\vec{n}~\cdot~\dot{\vec{r}}=0.$$
Because of this condition, not all components $(\dot{x},\dot{y},\dot{z})$ of the vector $\dot{\vec{r}}$ are independent. Since $\vec{n}$ is unital at every moment of time at least one of its components $(n_x,n_y,n_z)$ does \emph{not} vanish. Without loss of generality we can assume that at the moment $t$ the component $n_z\neq 0$. Then multiplying equation (\ref{metz}) by $n_z$ we get:
$$0=n_z \dot{\vec{r}}~\cdot~\dot{\vec{n}}=n_z(\dot{x}\dot{n}_x+\dot{y}\dot{n}_y)+\dot{n}_zn_z\dot{z}=n_z(\dot{x}\dot{n}_x+\dot{y}\dot{n}_y)+\dot{n}_z(-n_x\dot{x}-n_y\dot{y}),$$
where we have used the tangency to the distribution condition $n_z\dot{z}=-n_x\dot{x}-n_y\dot{y}$. This eventually shows that
\be (n_z\dot{n}_x-n_x\dot{n}_z)\dot{x}-(n_y\dot{n}_z-n_z\dot{n}_y)\dot{y}=0.\label{metz2}\ee
This is nothing but the condition for the vanishing of the $z$ component of the equation (\ref{metz1}). But the above analysis ensures that, on the open set in $\mathcal C$ in which $n_z\neq 0$, the conditions for vanishing of the two other components of (\ref{metz1}), which \emph{a'priori} may be different from the vanishing of the $z$ component, are actually equivalent to the condition (\ref{metz2}). If $n_z>0$ we can use our parametrization in which $(\vec{r},\vec{n})=(x,y,z,a,b)$. In this parametrization
$$\dot{\vec{n}}=-(\dot{f}a+f\dot{a})\vec{e}_x-(\dot{f}b+f\dot{b})\vec{e}_y+\dot{f}\vec{e}_z,\quad{\rm with}\quad f=\frac{1}{\sqrt{1+a^2+b^2}}.
$$
Hence
$$ n_z\dot{n}_x-n_x\dot{n}_z=-f^2\dot{a},\quad\quad\quad\quad n_y\dot{n}_z-n_z\dot{n}_y=f^2\dot{b},
$$
and the equation (\ref{metz2}) becomes equivalent to:
\be \dot{a}\dot{x}+\dot{b}\dot{y}=0.\label{nol}\ee
This equation can be interpreted as follows: a saucer performs an attacking mode aerobatic man{\oe}uvre if its trajectory $\gamma(t)=(x(t),y(t),z(t),a(t),b(t))$ in the configuration space is tangent to $\mathcal D$ and is a \emph{null curve} in a split-signature metric 
\be
g=2(\der x\der a+\der y\der b)\quad{\rm defined}~{\rm on}~ {\mathcal D}.\label{mett}\ee
 Note that the nullity condition (\ref{nol}) equips $\mathcal D$ with a \emph{conformal class of metrics} $[g]$ rather, then just a single metric $g$. Having defined this conformal class, one can forget about the physical definition of the attacking mode aerobatic man{\oe}uvre, as we defined it in terms of the geometry in the 3-space, and simply say that the \emph{a saucer performs attacking mode aerobatic man{\oe}uvre if and only if its trajectories in the configuration space are  tangent to} $\mathcal D$ \emph{ and null with respect to} $[g]$. We stress that the conformal class $[g]$ is only defined on $\mathcal D$. There, it can be represented in a coordinate free way by 
\be
g~=~\der \vec{r}~\cdot~\der \vec{n}.\label{mett1}\ee
Thus, when a pilot of a flying saucer is capable of performing attacking mode aerobatic moanouvers, the configuration space $\mathcal C$ of his saucer is equipped with a \emph{contact subconformal split signature geometry}. The word \emph{contact} refers to the contact distribution $\mathcal D$, and the words \emph{subconformal split signature} refers to the structure $[g]$ associated with the conformal class of a metric $g$, which on $\mathcal D$ has \emph{split signature} $(+,+,-,-)$. 
\subsection{Structure group of the contact subconformal structure}\label{secstrgr} 
Since in our parametrization the contact form $\omega^0$ can be chosen so that $$\omega^0=\der z -a \der x-b\der y,$$ then we have $\Omega=\der\omega^0=\der x\dz\der a+\der y\dz\der b$. This enables us to introduce one forms 
$$\omega^1=\der x,\quad\omega^2=\der y,\quad\omega^3=\der b,\quad\omega^4=\der a$$
such that
$$\begin{aligned}
g&=g_{ij}\omega^i\omega^j=2(\omega^1\omega^4+\omega^2\omega^3),\\
\Omega&=\tfrac12\Omega_{ij}\omega^i\dz\omega^j=\omega^1\dz\omega^4+\omega^2\dz\omega^3.\end{aligned}$$
Here, and in the following, the Latin indices $i,j,k,l$ run through the numbers 1,2,3,4.

Note that 
$$\omega^0\dz\omega^1\dz\omega^2\dz\omega^3\dz\omega^4\neq 0,$$
hence we have a coframe $(\omega^0,\omega^1,\omega^2,\omega^3,\omega^4)$ on ${\mathcal C}_>$.

The subconformal contact geometry $(M,{\mathcal D},[g],[\Omega])$ is then defined on $M$ by classes of metrics $[g]$ on $\mathcal D$ related to $g$ via the equivalence relation
\be  \bar{g}\sim g\quad{\rm  iff}\quad \bar{g}=f_1 g,\quad f_1>0,\label{gg}\ee
and by classes of 2-forms $[\Omega]$ on $\mathcal D$ related to $\Omega$ via the  equivalence relation
\be  \bar{\Omega}\sim\Omega\quad{\rm iff}\quad \bar{\Omega}=f_2\Omega,\quad f_2> 0.\label{omom}\ee 

Since both $g$ and $\Omega$ are only defined on the distribution $\mathcal D$, and since they, as well as $\omega^0$, are only defined up to scales, then the coframe $(\omega^0,\omega^1,\omega^2,\omega^3,\omega^4)$ is defined up to the following transformations:
$$\begin{aligned}
\omega^0&\mapsto \bar{\omega}^0=\alpha \omega^0\\
\omega^i&\mapsto \bar{\omega}^i=\alpha^i_{~j}\omega^j+ \beta^i\omega^0
\end{aligned}
$$
with $\alpha\det(\alpha^i_{~j})\neq 0$ and with
\be 
\alpha^k_{~i}\alpha^l_{~j}g_{kl}=c_1g_{ij}\quad\&\quad\alpha^k_{~i}\alpha^l_{~j}\Omega_{kl}=c_2\Omega_{ij},\quad\quad c_1c_2\neq 0.\label{pres}\ee
The $4\times 4$ matrices $(\alpha^i_{~j})$ must satisfy (\ref{pres}) to preserve $g$ and $\Omega$ up to scales on $\mathcal D$. Before these equations were imposed, at each point of $\mathcal C$, the matrices  $(\alpha^i_{~j})$ have values in the structure group of $\mathcal D$, which is just the full general linear group $G=\glg(4,\bbR)$. Equations (\ref{pres}) reduce this group to a subgroup $G_0$ which we determine now. We do it by looking at the Lie algebra of $G_0$.

We write the matrix $(\alpha^i_{~j})$ in the form $(\alpha^i_{~j}(s))=(\exp(s Y)^i_{~j})$ and then take the derivative $\frac{\der}{\der s}_{|s=0}$ on both sides of the equations \ref{pres}. This results in the following linear equations for the matrices $(Y^i_{~j})$ generating the Lie algebra of $G_0$:
\be 
Y^k_{~i}g_{kj}+Y^k_{~j}g_{ik}=f_1g_{ij}\quad\&\quad Y^k_{~i}\Omega_{kj}+Y^k_{~j}\Omega_{ik}=f_2\Omega_{ij}.\label{presa}\ee
It follows that the \emph{space of solutions} to the equations (\ref{presa}) \emph{is 5-dimensional}. It is equipped with the structure of a 5-dimensional Lie algebra with the commutator as the usual commutator of $4\times 4$ matrices. This Lie algebra is the Lie algebra $\mathfrak{g}_0$ of the reduced by (\ref{presa}) structure group $G_0$. We describe it in full detail in the following proposition.
\begin{proposition}\label{prY}
The Lie algebra of $\mathfrak{g}_0$ that preserves the forms $g$ and $\Omega$ up to scales is:
$$\mathfrak{g}_0=\Span_\bbR(Y_1,Y_2,Y_3,Y_4,Y_5),$$
where the generators $Y_A$ are the matrices:
\be\begin{aligned}
 Y_1&=\bma 1&0&0&0\\
         0&-1&0&0\\
         0&0&1&0 \\
         0&0&0&-1 \ema,\quad 
Y_2=\bma 0&1&0&0\\
         0&0&0&0\\
                  0&0&0&-1 \\
         0&0&0&0 \ema,\quad 
Y_3=\bma 0&0&0&0\\
         1&0&0&0\\
                 0&0&0&0 \\
         0&0&-1&0 \ema,\\&\\
 &\quad\quad\quad\quad Y_4=\bma 1&0&0&0\\
         0&1&0&0\\
         0&0&-1&0 \\
                  0&0&0&-1 \ema,\quad 
Y_5=\bma 1&0&0&0\\
         0&1&0&0\\
                  0&0&1&0 \\
         0&0&0&1 \ema={\rm Id}.
\end{aligned}\label{bas4}
\ee 
The commutation relations are
$$[Y_1,Y_2]=2Y_2,\quad [Y_1,Y_3]=-2Y_3,\quad [Y_2,Y_3]=Y_1,$$
and modulo the antisymmetry all the other commutators vanish.

Thus,
$$\mathfrak{g}_0=\slg(2,\bbR)\oplus\bbR^2\subset\glg(4,\bbR).$$
\end{proposition}
\begin{remark}\label{redu}
Note that the obtained 4-dimensional representation of the reduced structure Lie algebra $\mathfrak{g}_0$ of the distribution $\mathcal D$ is \emph{reducible}!  
The representation space $\bbR^4$ decomposes into $\bbR^4=\bbR^2\oplus\bbR^2$, with the two $\mathfrak{g}_0$-invariant subspaces given as eigenspaces of a $\mathfrak{g}_0$-invariant operator
$$K=\bma 1&0&0&0\\
         0&1&0&0\\
         0&0&-1&0 \\
                  0&0&0&-1 \ema=Y_4$$
Since $[K,Y_i]=0$ for all $i=1,2,\dots 5$, and since $K^2={\rm Id}$, then the two $\mathfrak{g}_0$-invariant vector subspaces of $\bbR^4$ are eigenspaces of $K$ with the respective eigenvalues +1 and -1. Note that all  $Y_i$s, $i=1,2,\dots,5$, have block diagonal form 
$$Y_i=\bma A_i&0\\0&B_i\ema,$$
where $A_i$ and $B_i$ are $2\times 2$ real matrices!
\end{remark}
By exponentiation of the generators $Y_i$ of $\mathfrak{g}_0$ given in  Proposition \ref{prY} we obtain the group elements of ${\bf G}_0$. This leads to the following ${\bf G}$-structure on the configuration space $\mathcal C$ of a flying saucer capable of performing attacking mode aerobatics. 
\begin{definition}
The subconformal contact geometry $({\mathcal C}, [g], [\Omega])$ on the configuration space $\mathcal C$ of a flying saucer capable of performing attacking mode aerobatics consists of a coframe $(\omega^\mu)$, $\mu=0,1,\dots,4$, on ${\mathcal C}$, with local representative 
$$ (\omega^0,\omega^1,\dots,\omega^4)=(\der z -a \der x - b \der y,\der x,\der y, \der b,\der a),$$
which is given up to the following transformations 
\be\bma\bar{\omega}^0\\\bar{\omega}^1\\\bar{\omega}^2\\\bar{\omega}^3\\\bar{\omega}^4\ema=\bma 
s_6&0&0&0&0\\
s_{7}&s_1&s_2&0&0\\
s_{8}&s_3& s_4&0&0\\
s_{9}&0&0&s_5 s_1&-s_5 s_2\\
s_{10}&0&0&-s_5s_3&s_5 s_4
\ema\bma\omega^0\\\omega^1\\\omega^2\\\omega^3\\\omega^4\ema.\label{the}\ee
The functions $s_1,s_2,\dots, s_{10}$ appearing here must satisfy $(s_1s_4-s_2s_3)s_5 s_6\neq 0$ at each point of ${\mathcal C}$. The classes $[g]$ and $[\Omega]$ are then represented by the metric $\bar{g}=2(\bar{\omega}^1\bar{\omega}^4+\bar{\omega}^2\bar{\omega}^3)$ and by the 2-form $\bar{\Omega}= \bar{\omega}^1\dz\bar{\omega}^4+\bar{\omega}^2\dz\bar{\omega}^3$, respectively. 
\end{definition}

\subsection{Legendrean contact structure}
We now return to Remark \ref{redu}, in which we observed that the subconformal contact structure $([g],[\Omega])$ defines the Lie algebra $\mathfrak{g}_0$ in the \emph{reducible} representation, with the operator $K$ which at every point $(\vec{r},\vec{n})\in\mathcal C$ splits the 4-dimensional vector space ${\mathcal D}_{(\vec{r},\vec{n})}$ of the distribution $\mathcal D$ onto 2-dimensional eigenspaces of $K$. The invariant operator $K$ is entirely defined in terms of the classes $([g],[\Omega])$ defining the subconformal contact structure on $\mathcal C$. To see this in a more geometric way than in the algebraic one described by Remark \ref{redu}, choose representatives $\bar{g}$ and $\bar{\Omega}$ of their respective conformal classes $[g]$ and $[\Omega]$ given by (\ref{gg}) and (\ref{omom}). Define $\tilde{K}$ on $\mathcal D$  via the equation:
\be \bar{g}(\tilde{K}\cdot,\cdot)=\bar{\Omega}(\cdot,\cdot)\label{her}\ee
or, since $g$ is invertible on $\mathcal D$, via:
$$\tilde{K}^i_{~j}=\bar{g}^{ik}\bar{\Omega}_{kj}.$$
Then a short calculation shows that in a basis $(\omega^i)$, $i=1,2,3,4$, as defined in Section \ref{secstrgr}, the operator $\tilde{K}$ is: 
$$\tilde{K}=\frac{f_2}{f_1}Y_4.$$
In particular, its square is proportional to the identity, $\tilde{K}^2=\frac{f_2^2}{f_1^2} Id$. This defines $K$, modulo a sign, by the condition that $\tilde{K}^2=Id$ and $K:=\tilde{K}$. For this $f_2^2=f_1^2$, which always can be solved.

Now, because $K:{\mathcal D}\to {\mathcal D}$ is such that $K^2={\rm Id}$, it makes the split: 
$${\mathcal D}={\mathcal D}^+\oplus{\mathcal D}^-,$$
with 
$$K{\mathcal D}^\pm=\pm {\mathcal D}^\pm,$$
and one can check that both ${\mathcal D}^\pm$ have rank 2. It is easy to see that the metric $g$, when restricted to each of the spaces ${\mathcal D}^\pm$ separately, identically vanishes. Thus these spaces are \emph{totally null} with respect to $g$. Also $\Omega$ restricted to each of the spaces ${\mathcal D}^\pm$ identically vanishes. Thus both ${\mathcal D}^\pm$, as having rank 2, are \emph{Lagrangean} in $\Omega$. This equips the configuration space $\mathcal C$ of a flying saucer with the so called \emph{Legendrean contact structure}, see \cite{Cap}, Section 4.2.3.  

For the reader not familiar with the monograph \cite{Cap} we recall that a $(2n+1)$-dimensional manifold $M$ is equipped with a \emph{Legendrean contact structure} $({\mathcal D},{\mathcal D}^\pm)$ iff
\begin{itemize}
\item ${\mathcal D}$ is the annihilator of a \emph{contact} one form form $\omega^0$ on $M$, i.e. a one form such that $$\overbrace{\der\omega^0\dz\dots\dz\der\omega^0}^{2n}\dz\omega^0\neq 0,$$
\item we have distinguished split: ${\mathcal D}={\mathcal D}^+\oplus{\mathcal D}^-$,
\item and the spaces ${\mathcal D}^\pm$ are \emph{Lagrangean} i.e. each of them has rank $n$ and the form $\Omega=(\der \omega^0)_{|\mathcal D}$ when restricted to them identically vanishes, $\Omega_{|{\mathcal D}^\pm}\equiv 0$.
\end{itemize}

In special situations Legendrean contact structures $({\mathcal D},{\mathcal D}^\pm)$ can originate from \emph{split signature} conformal structures $[g]$  on $\mathcal D$. This happens when in the class $[g]$ and in the conformal class of symplectic forms $[\Omega=\der\omega^0]$ on $\mathcal D$ there exist respective $g'$ and $\Omega'$ such that the linear operator $K:{\mathcal D}\to \mathcal D$ defined via: 
$$K^i_{~j}={g'}\phantom{}^{ik}\Omega'_{kj}$$
squares to the identity on $\mathcal D$, $K^2=Id_{\mathcal D}$, and has the plus/minus-one-eigenvalues with the eigenspaces of the same dimension. In such case we will say that the split signature conformal structure $[g]$ on $\mathcal D$ is \emph{compatible} with the contact structure $[\omega^0={\mathcal D}^\perp]$ on $\mathcal D$.

In this sense, a flying saucer equipped with the attacking mode aerobatic man{\oe}uvre naturally acquires a conformal structure $[g]$ on its velocity distribution $\mathcal D$ which is compatible with the contact structure given by $\mathcal D$.
  
From now on, when talking about Legendrean contact structures we will restrict to the structures coming from pairs $({\mathcal D},[g])$ with the split signature conformal structure $[g]$ compatible with the contact structure of $\mathcal D$. We therefore may replace the symbol $\mathcal D$ of the distribution with the symbol $[\omega^0]$, which denotes the class of a contact one form $\omega^0$ given on $M$ up to a scale and, instead of writing $(M,{\mathcal D},{\mathcal D}^\pm)$ for such Legendrean contact structures, we will write  $(M,[\omega^0],[g])$.

Often one asks about \emph{equivalences between Legendrean contact structures}. In particular, in our case of two such structures $(M,[\omega^0],[g])$ and $(\bar{M},[\bar{\omega}^0],[\bar{g}])$ \emph{are locally equivalent} iff there exists a local diffeomorphism $\phi:M\to\bar{M}$ such that
$$\begin{aligned}
&\phi^*\bar{\omega}^0=f\omega^0\\
&\phi^*\bar{g}=hg+\tau\omega^0
\end{aligned}$$
with nonvanishing function $f$ and $h$ on $M$ and with a certain one form $\tau$ on $M$. Local \emph{self-equivalences} for  $(M,[\omega^0],[g])$, are called local \emph{symmetries} of  $(M,[\omega^0],[g])$. 

Infinitesimal versions of local symmetries are \emph{infinitesimal symmetries}. These are vector fields $X$ on  $(M,[\omega^0],[g])$ such that 
\be\begin{aligned}
&{\mathcal L}_X\omega^0=p \omega^0\\
&{\mathcal L}_Xg=q g+\mu\omega^0
\end{aligned}\label{sym}\ee
with functions $p$, $q$ and with a certain one form $\mu$ on $M$. It follows that infinitesimal symmetries of a Legendrean contact structure  $(M,[\omega^0],[g])$ form a Lie algebra - a Lie algebra of symmetries of  $(M,[\omega^0],[g])$. It is also known that there exists locally \emph{non}equivalent Legendrean contact structures, and that among all Legendrean contact structures of a given dimension, there is a unique (modulo local equivalence) Legendrean contact structure with the \emph{highest dimension of its Lie algebra of Local symmetries}. This unique structure is called the \emph{flat} Legendrean contact structure, and in dimension $(2n+1)$ of $M$ its \emph{Lie algebra of symmetries is isomorphic} to the simple Lie algebra $\sla(2n,\bbR)$ of dimension $4n^2-1$.

\subsection{Legendrean contact structure for an attacking mode aerobatic monouver} 
Let us summarize our considerations about the geometry of the configuration space of a flying saucer that is equipped to perform attacking mode aerobatic man{\oe}uvres.
\begin{proposition}
The configuration space ${\mathcal C}$ of a flying saucer in attacking mode is naturally equipped with a 5-dimensional Legendrean contact structure $({\mathcal C},[\omega^0],[g])$. Here $\omega^0$ defines the contact distribution  ${\mathcal D}=(\omega^0)^\perp$, and $[g]$ is the man{\oe}uvre induced split signature conformal structure $[g]$ defined on ${\mathcal D}=(\omega^0)^\perp$ and compatible with the contact structure of $\mathcal D$. In the natural coordinates $(\vec{r},\vec{n})$ in ${\mathcal C}=\bbR^3\times\bbS^2$ the contact structure is defined via the contact one form  
$$\omega^0~=~\vec{n}~\cdot~\der\vec{r},$$
and the conformal structure $[g]$ is represented by $g$ with 
$$g~=~\big(~\der\vec{r}~\cdot~\der\vec{n}~\big)_{|\mathcal D}.$$
The attacking mode aerobatic man{\oe}uvre consists in a movement of a saucer along the trajectories in $\mathcal C$ tangent to $\mathcal D$ and null with respect to $[g]$. 
\end{proposition}
It is interesting to characterize the Legendrean contact structure $({\mathcal C},[\omega^0],[g])$ of a flying saucer with attacking mode aerobatic man{\oe}uvres among all the Legendrean contact structures in dimension 5. For this we determine the algebra of symmetries of $({\mathcal C},[\omega^0],[g])$. 

Working in our parametrization of ${\mathcal C}_>$ we want to find all vector 
fields $$X=A_1\partial_x+A_2\partial_y+A_3\partial_z+A_4\partial_a+A_5\partial_b$$ on ${\mathcal C}_>$
such that the equations (\ref{sym}) hold with 
$$\omega^0=\der z-a \der x-b\der y\quad\quad\quad {\rm and}\quad\quad\quad g=2(\der x\der a+\der y\der b).$$

Using Ian Anderson's wonderful Maple Differential Geometry Package we easily find 15 symmetries. We have the following proposition. 

\begin{proposition}\label{symprop}
The following 15 linearly independent vector fields are symmetries of the attacking mode aerobatics Legendrean contact structure $\Big(~{\mathcal C}_>0,~[\der z - a \der x-b \der y],~[\der x \der a+\der y\der b]~\Big)$:
$$\begin{aligned}
X_1&=z(x\partial_x+y\partial_y+z\partial_z)+(z-ax-by)(a\partial_a+b\partial_b)\\
X_2&=x(x\partial_x+y\partial_y+z\partial_z)+(z-ax-by)\partial_a\\
X_3&=-z\partial_y+b(a\partial_a+b\partial_b)\\
X_4&=-x\partial_y+b\partial_a\\
X_5&=-z\partial_x+a(a\partial_a+b\partial_b)\\
X_6&=-x\partial_x+a\partial_a\\
X_7&=x\partial_z+\partial_a\\
X_8&=y(x\partial_x+y\partial_y+z\partial_z)+(z-ax-by)\partial_b\\
X_9&=-y\partial_x+a\partial_b\\
X_{10}&=x\partial_x+z\partial_z+b\partial_b\\
X_{11}&=y\partial_z+\partial_b\\
X_{12}&=x\partial_x+y\partial_y+z\partial_z\\
X_{13}&=\partial_y\\
X_{14}&=\partial_x\\
X_{15}&=\partial_z.
\end{aligned}
$$ 
\end{proposition}
\begin{remark}
It is worth nothing that out of these 15 symmetries 8 of them, namely $$X_4,X_6,X_7,X_9,X_{11},X_{13},X_{14},X_{15},$$ preserve $\omega^0$ and $g$ \emph{exactly}. For these $X$ we have ${\mathcal L}_X\omega^0= 0$ and ${\mathcal L}_Xg= 0$. Vectors $X_{10}$ and $X_{12}$ correspond to \emph{homotheties} for $\omega^0$ and $g$. For them we have  ${\mathcal L}_X\omega^0=\omega^ 0$ and ${\mathcal L}_Xg=g$. The remaining 5 symmetries, namely $X_1,X_2,X_3,X_5,X_8$, are \emph{proper conformal contact Killings} for $g$. By this we mean that both $q$ and $\mu$ appear as nonzero factors in the equations (\ref{sym}) for these 5 vector fields $X$. 
\end{remark}
 
Since Proposition \ref{symprop} explicitly gives 15 symmetries of  $({\mathcal C}_>0,[\der z - a \der x-b \der y],$
$[\der x \der a+\der y\der b])$, and 15 is the dimension of the Lie algebra $\sla(4,\bbR)$ which is the algebra of symmetries of the flat Legendrean structure in dimension 5, one immediately suspects that the attacking mode aerobatics Legendrean conformal structure is flat. That this is realy the case requires a bit of structural theory. Anderson's Maple Package shows that the algebra defined by $(X_1,X_2,\dots,X_{15})$ is simple and finds the Cartan subalgebra, which turns out to have rank 3. Then further classification procedure shows that its Cartan matrix is of type $A$. To realize that the symmetry Lie algebra is the $\sla(4,\bbR)$ real form of $A_3$ we calculate the signature of the Killing form. Comparing this with the signatures of the Killing forms of the standard models of real forms of $A_3$ realized as $4\times 4$ matrices we eventually prove the following theorem:
\begin{theorem}
The configuration space $({\mathcal C})$ of 
the flying saucer equipped to perform standard aerobatic man{\oe}uvres is the flat 5-dimensional Legendrean contact structure. As such it has $\sla(4,\bbR)$ Lie algebra as the algebra of its symmetries. 
\end{theorem} 
\section{Landing mode aerobatics - 5-dimensional CR structure}\label{sa2}  There is another man{\oe}uvre of a flying saucer which mathematically is very similar to the attacking mode aerobatic man{\oe}uvre. 

We say that a \emph{landing mode aerobatic man{\oe}uvre} is a movement of the flying saucer which respects the following rule: \emph{at every moment the speed} $\dot{\vec{r}}(t)$ \emph{of the saucer in 3-space is \underline{orthogonal} to both} $n(t)$ \emph{and the line}~$l(t)$. This is the same as saying that \emph{the velocity} $\dot{\vec{r}}(t)$ \emph{of the saucer in 3-space is \underline{parallel} to the line} $l^\perp(t)$.

Once the pilot of a saucer learns this man{\oe}uvre, we have a similar situation as with the attacking mode man{\oe}uvre case: the configuration space $({\mathcal C},{\mathcal D})$ of his saucer is equipped with an additional structure. This structure is as follows.

\subsection{Five dimensional Cauchy-Riemann structure} The \emph{landing mode aerobatics rule} says that at every moment of time $t$ we have:
$$\dot{\vec{r}}(t) ~||~ l^\perp(t).$$
In our coordinates $(x,y,z,a,b)$ this means that the 3-space speed vector $(\dot{x},\dot{y},\dot{z})$ of the saucer is parallel in 3-space to the vector $\Big((1+b^2)\dot{a}-a b \dot{b},(1+a^2)\dot{b}-ab \dot{a},a \dot{a}+b \dot{b}\Big)$ tangent to $l^\perp(t)$. This means that the cross product of these two vectors is zero:
\be
(\dot{x},\dot{y},\dot{z})~\times~\Big((1+b^2)\dot{a}-a b \dot{b},(1+a^2)\dot{b}-ab \dot{a},a \dot{a}+b \dot{b}\Big)~=~0.\label{metr2}\ee
And now, again, the \emph{a priori} three conditions, which are implied by this equation reduce to only one scalar equation owing to:
$$\dot{z}-a\dot{x}-b\dot{y}=0.$$
This single scalar equation equivalent to (\ref{metr2}) reads:
$$\Big((1+a^2)\dot{b}-ab\dot{a}\Big)\dot{x}-\Big((1+b^2)\dot{a}-ab\dot{b})\Big)\dot{y}=0.$$
This defines a \emph{conformal metric} 
\be
\hat{g}=2\Big((1+a^2)\der b-ab\der a\Big)\der x-2\Big((1+b^2)\der a-ab\der b)\Big)\der y\label{met}\ee
on the distribution $\mathcal D$. Clearly, it has $[+,+,-,-]$ signature on $\mathcal D$. 

So now again we realize the aerobatic trajectories, the \emph{landing mode} ones this time, as \emph{null curves} in a \emph{class of split signature conformal metrics on the distribution}. There is, however an important difference:

Let us use the same basis forms on ${\mathcal C}_>$ as in Section \ref{secstrgr}. Let $e_i$ be the duals to $\omega^i$, i.e. $e_i\hook\omega^j=\delta_i^{~j}$, $e_i\hook\omega^0=0$. If we consider, as before, an operator $K=K^i_{~j}e_i\otimes\omega^j$, defined by the metric $$g=g_{ij}\omega^i\omega^j$$ from (\ref{met}), and the
symplectic form on the distribution
$$\Omega=\der \omega^0=\tfrac12\Omega_{ij}\omega^i\dz\omega^j=\der x\dz\der a+\der y \dz\der b,$$
via
$$K^i_{~j}=g^{ik}\Omega_{kj},$$
we discover that now:
$$K^2=\frac{1}{1+a^2+b^2}\bma-1&0&0&0\\0&-1&0&0\\0&0&-1&0\\0&0&0&-1\ema.$$
Thus, after appropriate rescalling of $g$ and $\Omega$ we have now an operator $K$ on the contact distribution $\mathcal D$ that squares to the \emph{minus} identity,
$$K^2=-\id.$$
This equips $({\mathcal C},{\mathcal D})$ with a 5-dimensional \emph{CR structure} of hypersurface type. To check what is the signature of its Levi form we write down the rescaled $K$ explicitly, as
$$K=\frac{1}{\sqrt{1+a^2+b^2}}
  \bma -ab&-(1+b^2)&0&0\\
  1+a^2&ab&0&0\\
  0&0&ab&-(1+a^2)\\0&0&1+b^2&-ab\ema.$$
  Its $\pm i$ eigenspaces ${\mathcal D}_\pm$ (which are subspaces of the \emph{complexification} ${\mathcal D}^\bbC$ of $\mathcal D$) are
  $${{\mathcal D}_\pm}=\Span\Big(\,\pm i(1+a^2)\partial_a+(\sqrt{1+a^2+b^2}\pm iab)\partial_b,\,\pm i(1+b^2)\partial_x+(\sqrt{1+a^2+b^2}\mp iab)\partial_y\,\Big).$$
  Introducing a basis $(Z_1,Z_2,\bar{Z}_1,\bar{Z}_2)$ in the complexification of $\mathcal D$ as:
  $$\begin{aligned}
    &Z_1= i(1+a^2)\partial_a+(\sqrt{1+a^2+b^2}+ iab)\partial_b,\quad Z_2= i(1+b^2)\partial_x+(\sqrt{1+a^2+b^2}- iab)\partial_y\\
    &\bar{Z}_1= -i(1+a^2)\partial_a+(\sqrt{1+a^2+b^2}-iab)\partial_b,\quad \bar{Z}_2=-i(1+b^2)\partial_x+(\sqrt{1+a^2+b^2}+ iab)\partial_y,
  \end{aligned}$$
  with the respective dual basis $(\sigma^1,\sigma^2,\bar{\sigma}^1,\bar{\sigma}^2)$, $Z_A\hook\sigma^b=\delta_A^{~B}$, we see that the symplectic form $\Omega$ on the distribution $\mathcal D$ reads:
  $$\Omega=L_{A\bar{B}}\sigma^A\dz\sigma^{\bar{B}}=-C\sigma^1\dz\bar{\sigma}^2-\bar{C}\bar{\sigma}^1\dz\sigma^2,$$
  with
  $$C=2(1+a^2+b^2+iab\sqrt{1+a^2+b^2}).$$
  The $2\times2$ complex valued matrix $L=(L_{A\bar{B}})$ represents the \emph{Levi form} of the corresponding CR structure, and it reads
  $$L_{A\bar{B}}=\bma 0&-C\\-\bar{C}&0\ema.$$
  It obviously has signature $(1,1)$.

With this, we conclude that the configuration space  $({\mathcal C},{\mathcal D})$ of a flying saucer equipped with the landing mode aerobatic man{\oe}uvre is a 5-dimensional \emph{CR structure} of hypersurface type with Levi form of signature $(1,1)$. 
    
We are thus in an interesting situation:
\begin{itemize}
\item the \emph{attacking mode aerobatic man{\oe}uvre} equips the configuration space $({\mathcal C},{\mathcal D})$ of a flying saucer with \emph{5-dimensional Legendrean contact structure},
\item but the \emph{landing mode aerobatic man{\oe}uvre} equips this configuration space with a \emph{CR  structure} (to see that it is integrable refer to \cite{dynamics}). 
\end{itemize}
We recall that the Legendrean structure associated with the attacking mode man{\oe}uvre is flat, and as such has 15-dimensional group of symmetries. It is therefore natural to ask about the dimension of the Lie algebra of symmetries of the landing mode CR structure. In full analogy with the previous case we have the following proposition.
\begin{proposition}
The symmetry algebra of the CR structure $$\Big(~{\mathcal C}_>0,~[\der z - a \der x-b \der y],~[\Big((1+a^2)\der b-ab\der a\Big)\der x-\Big((1+b^2)\der a-ab\der b)\Big)\der y]~\Big)$$ is the 15-dimensional Lie algebra $\sua(2,2)$. It is spanned by the following 15 vector fields on ${\mathcal C}_>$:
$$\begin{aligned}
X_1&=-z(x\partial_x+y\partial_y)+\tfrac12(x^2+y^2-z^2)\partial_z+\big((1+a^2)x+aby\big)\partial_a+\big((1+b^2)y+abx\big)\partial_b\\
X_2&=\tfrac12(y^2+z^2-x^2)\partial_x-x(y\partial_y+z\partial_z)-\big((1+a^2)z-by\big)\partial_a-a(bz+y)\partial_b\\
X_3&=y(x\partial_x+z\partial_z)+\tfrac12(y^2-x^2-z^2)\partial_y+b(az+x)\partial_a+\big((1+b^2)z-ax\big)\partial_b\\
X_4&=\tfrac12\frac{x^2+y^2+z^2}{\sqrt{a^2+b^2+1}}(\partial_z-a\partial_x-b\partial_y)+\sqrt{a^2+b^2+1}\Big((az+x)\partial_a+(bz+y)\partial_b\Big)\\
X_5&=-z\partial_x+x\partial_z+(a^2+1)\partial_a+ab\partial_b\\
X_6&=-z\partial_y+y\partial_z+ab\partial_a+(b^2+1)\partial_b\\
X_7&=y\partial_x-x\partial_y+b\partial_a-a\partial_b\\
X_8&=\frac{x}{\sqrt{a^2+b^2+1}}(\partial_z-a\partial_x-b\partial_y)+\sqrt{a^2+b^2+1}\partial_a\\
X_9&=\frac{z}{\sqrt{a^2+b^2+1}}(\partial_z-a\partial_x-b\partial_y)+\sqrt{a^2+b^2+1}(a\partial_a+b\partial_b)\\
X_{10}&=\frac{y}{\sqrt{a^2+b^2+1}}(\partial_z-a\partial_x-b\partial_y)+\sqrt{a^2+b^2+1}\partial_b\\
X_{11}&=\frac{1}{\sqrt{a^2+b^2+1}}(\partial_z-a\partial_x-b\partial_y)\\
X_{12}&=x\partial_x+y\partial_y+z\partial_z\\
X_{13}&=\partial_x\\
X_{14}&=\partial_y\\
X_{15}&=\partial_z.
\end{aligned}
$$ 
\end{proposition} 
\begin{proof}
We simply solved the infinitesimal symmetry equations (\ref{sym}) for $X$ with 
$\omega^0=\der z - a \der x-b \der y$ and $g=\Big((1+a^2)\der b-ab\der a\Big)\der x-\Big((1+b^2)\der a-ab\der b)\Big)\der y$. We found a 15-dimensional space of solutions with the basis $(X_1, X_2,\dots, X_{15})$ as in the statement of the proposition.
\end{proof}

\section{$G_2$ aerobatics}
In this section we teach a pilot of a flying saucer new man{\oe}uvres that will equip configuration space $({\mathcal C},{\mathcal D})$ of the saucer with \emph{exceptional} $G_2$ \emph{contact geometry in dimension five}. For this we need some preparations.  
\subsection{Irreducible $\glg(2,\bbR)$ in dimension four}\label{sec:gl2r}
The real vector space $\bbR^4$ can be identified \cite{Br1} with a symmetric third tensorial power of $\bbR^2$:
$$\bbR^4\equiv \bgs^3\bbR^2.$$
The identification map identifies any \emph{4-vector} $X^i=(X^1,X^2,X^3,X^4)\in\bbR^4$ with a \emph{3-spinor} $\Psi^{ABC}=\Psi^{(ABC)}=(\Psi^{111},\Psi^{112},\Psi^{122},\Psi^{222})\in \bgs^3\bbR^2$ via:
$$\bbR^4~~\ni \quad\boxed{ X^i\equiv\Psi^{ABC}}\quad \in~~\bgs^3\bbR^2,$$
so that we have
$$X^1\equiv \Psi^{111},\quad X^2\equiv \Psi^{112},\quad X^3\equiv \Psi^{122},\quad X^4\equiv \Psi^{222}.$$ 
Working with the third symmetric tensorial power of $\bbR^2$ rather than with $\bbR^4$ enables one, by use of a fixed volume form $\epsilon_{AB}=-\epsilon_{BA}$ in $\bbR^2$, to equip $\bbR^4$ with tensors which are $GL(2,\bbR)$ invariant. Indeed, given a vector $X^i\equiv\Psi^{ABC}$ we can associate with it a natural endomorphism $L(X)$ of $\bbR^2$ via:
$$L^A_{~H}(X)=\Psi^{ABC}\Psi^{DEF}\epsilon_{CD}\epsilon_{BE}\epsilon_{FH}.$$
Explicitly the matrix of this endomorphism reads:
$$\Big(L^A{}_H(X)\Big)=\bma X^2X^3-X^1X^4&-2(X^2)^2+2X^1X^3\\2(X^3)^2-2X^2X^4&-X^2X^3+X^1X^4\ema.$$
It defines, through its determinant, a \emph{direction} $\Span_\bbR[\ten]$ of \emph{a symmetric rank 4 tensor} $\ten_{ijkl}=\ten_{(ijkl)}$. This is given by:
\be\begin{aligned}
  \ten(X,X,X,X):=&\,{\rm det}\Big(L(X)\Big)\\
  =&\,3(X^2)^2(X^3)^2-4X^1(X^3)^3-4(X^2)^3X^4+6X^1X^2X^3X^4-(X^1)^2(X^4)^2\\=&\,\tfrac{1}{24}\ten_{ijkl}X^iX^jX^kX^l.\end{aligned}\label{ten4}\ee
There is a natural $\glg(2,\bbR)$ action $\tilde{\rho}$ on tensors $T_{i_1i_2\dots i_r}$ from $\bgt^r\bbR^4$. This is induced by the irreducible representation of $\glg(2,\bbR)$ in $\bbR^4$ given by:
$$\rho^i{}_j(\alpha)X^j:=\alpha^A{}_{A_1}\alpha^B{}_{B_1}\alpha^C{}_{C_1}\Psi^{(A_1B_1C_1)},\quad \alpha=(\alpha^A{}_B)\in\glg(2,\bbR).$$
Explicitly, the linear action $\tilde{\rho}$ is given in terms of the representation $\rho$ via:
$$\tilde{\rho}{}^{j_1j_2\dots j_r}{}_{i_1i_2\dots i_r}(\alpha)T_{j_1j_2\dots j_r}:=T_{j_1j_2\dots j_r}\rho(\alpha^{-1}){}^{j_1}{}_{i_1}\rho(\alpha^{-1}){}^{j_2}{}_{i_2}\dots \rho(\alpha^{-1}){}^{j_r}{}_{i_r}.$$
It follows from the construction that under this $\glg(2,\bbR)$ action, tensor $\ten_{ijkl}$ transforms up to a scale, $$\tilde{\rho}{}^{i_1j_1k_1l_1}{}_{ijkl}(\alpha)\ten_{i_1j_1k_1l_1}=f(\alpha)\ten_{ijkl},$$ i.e. the direction of this tensor $\Span_\bbR[\ten_{ijkl}]$ is $\glg(2,\bbR)$ invariant.

Although the matrix entries $L^A{}_H(X)$ of the endomorphism $L(X)$ are \emph{not} invariant with respect to $\glg(2,\bbR)$ action, it follows that the span  $\Span_\bbR[g^1,g^2,g^3]$, of the three symmetric bilinear forms $g^1{}_{ij},g^2{}_{ij},g^3{}_{ij}$ in $\bbR^4$ defined by these entries via:
$$\begin{aligned}
  g^1(X,X):=&\,X^1X^3-(X^2)^2&=\tfrac12g^1{}_{ij}X^iX^j\\
  g^2(X,X):=&\,(X^3)^2-X^2X^4&=\tfrac12g^2{}_{ij}X^iX^j\\
  g^3(X,X):=&\,X^1X^4-X^2X^3&=\tfrac12g^3{}_{ij}X^iX^j,
  \end{aligned}$$    
forms a $\glg(2,\bbR)$ invariant module. 

In addition, we also have a $\glg(2,\bbR)$\emph{-invariant direction} $\Span_\bbR[\omega]$ of \emph{a 2-form}. With the identifications $X^i=\Psi^{ABC}$ and $Y^i=\Phi^{ABC}$ this 2-form is given by:
\be\begin{aligned}
  \omega(X,Y):=&\,\Psi^{ABC}\Phi^{DEF}\epsilon_{CD}\epsilon_{BE}\epsilon_{AF}\\=&\,X^1Y^4-X^4Y^1-3X^2Y^3+3X^3Y^2\\=&\,\tfrac12\omega_{ij}X^iY^j.\end{aligned}\label{2form}\ee
In particular, $\omega(X,X)=L^A{}_{A}(X)=0$ and $$\omega\dz\omega\neq 0.$$

In the following we call the tensor $\ten$ defined in (\ref{ten4}) and the 2-form $\omega$ defined in (\ref{2form}) the \emph{structural tensors}. 

Given the structural tensor $\ten$ consider its (conformal) stabilizer $$S=\{A\in\glg(4,\bbR)\,\,{\rm s.t.}\,\,\ten(AX,AX,AX,AX)=f(A)\ten(X,X,X,X)\}$$ in $\glg(4,\bbR)$. We have the following proposition \cite{nurgl2}:
    \begin{proposition}
      The stabilizer $S$ of the structural tensor $\ten$ is $\glg(2,\bbR)$ in 4-dimensional irreducible representation. Moreover,
      $$S=\glg(2,\bbR)\subset \cspg(\omega)\subset\glg(4,\bbR),$$
i.e. $\ten$ reduces the structure group $\glg(4,\bbR)$ to the irreducible $\glg(2,\bbR)$ via the conformal symplectic group $$\cspg(\omega)=\{A\in\glg(4,\bbR)\,\,{\rm s.t.}\,\,\omega(AX,AY)=h(A)\omega(X,Y)\}$$ preserving the structural 2-form $\omega$.  
\end{proposition}
    \subsection{$\glg(2,\bbR)$ null vectors in dimension four}
    It is interesting to consider those nonzero vectors $X$ in $\bbR^4$ which are \emph{null} with respect to the structural tensor $\ten$. It is also interesting to consider those vectors $0\neq X\in\bbR^4$ which are simultaneously null with respect to all three bilinear forms $g^1$, $g^2$ and $g^3$, which define the $\glg(2,\bbR)$ invariant module  $\Span_\bbR[g^1,g^2,g^3]$. Let us denote by
    $$\mathcal T=\{{\rm dir}(X)\in \bbR^4\,\, {\rm s.t.}\,\,g^1(X,X)=g^2(X,X)=g^3(X,X)=0\}$$
the set of all directions in $\bbR^4$ that are null with respect to all three bilinear forms $g^1$, $g^2$ and $g^3$, and
    by
    $$TV({\mathcal T})=\{{\rm dir}(X)\in \bbR^4\,\, {\rm s.t.}\,\,\ten(X,X,X,X)=0\}$$
    the set  of all directions in $\bbR^4$ that are null with respect to the structural tensor $\ten$.

By elementary arguments we see that the dimension of the set $\mathcal T$ is equal to one, $\dim({\mathcal T})=1$, and that the dimension of the set $TV({\mathcal T})$ is equal to two, $\dim(TV({\mathcal T}))=2$. 
    
    To find a description of open sets in $\mathcal T$ and $TV({\mathcal T})$  we consider elements of $\bgs^3\bbR^2$ of the form $\Psi^{ABC}=\xi^{(A}\eta^B\tau^{C)}$, where $\xi,\eta,\tau\in\bbR^2$. We call $\xi,\eta$ and $\tau$ \emph{principal spinors} for $\Psi$. We may have $\Psi$'s with all three principal spinor \emph{directions} coinciding: $\Psi^{ABC}=v \xi^A\xi^B\xi^C$ - such $\Psi$s we call \emph{of type} $N$, with two principal spinor \emph{directions} coinciding: $\Psi^{ABC}=\xi^{(A}\xi^B\tau^{C)}$ - such $\Psi$s we call \emph{of type II}, or we have \emph{algebraically general} $\Psi$s, where all the principal spinor directions ${\rm dir}(\xi)$, ${\rm dir}(\eta)$ and ${\rm dir}(\tau)$ are distinct. 

    Let us consider vectors $X\in \bbR^4$ corresponding to $\Psi\in\bgs^3\bbR^2$ of type $N$. We thus have ${\rm dir}(X^i)={\rm dir}(\xi^A\xi^B\xi^C)$ for some $\xi\in\bbR^2$. Explicitly we can represent such a direction by a vector:
    $$(X^1,X^2,X^3,X^4)=\Big((\xi^1){}^3,(\xi^1){}^2\xi^2,\xi^1(\xi^2){}^2,(\xi^2)^3\Big).$$
    Because $\Psi$ in $\bgs^3\bbR^2$ representing this $X$ is $\Psi^{ABC}=\xi^A\xi^B\xi^C$, we have
    $$L^A{}_H(X)=\Psi^{ABC}\Psi^{DEF}\epsilon_{CD}\epsilon_{BE}\epsilon_{FH}=\xi^A\xi^B\xi^C\xi^D\xi^E\xi^F\epsilon_{CD}\epsilon_{BE}\epsilon_{FH}=0,$$
    due to the antisymmetry of $\epsilon$. Thus, $$g^1(X,X)=g^2(X,X)=g^3(X,X)=0,$$
    as well as $$\ten(X,X,X,X)=\det(L(X))=\det(0)=0.$$
    This proves that every vector $X\in\bbR^4$ corresponding to $\Psi\in\bgs^3\bbR^2$ of type $N$ is null with respect to the module $\Span_\bbR[g^1,g^2,g^3]$ and is also null with respect to the structural tensor $\ten$. Parametrizing $\xi$ as
    $$(\xi^1,\xi^2)=(u,ut),$$
    we obtain
$$(X^1,X^2,X^3,X^4)=\Big(u^3,u^3t,u^3t^2,u^3t^3),$$    
    i.e. that every direction of a null vector $X$ corresponding to the spinor $\Psi\in\bgs^3\bbR^2$ of type $N$ is given by
     \be {\rm dir}(X)=(1,t, t^2, t^3)^T,\label{tc}\ee
i.e. it lies on the \emph{twisted cubic}
    $$\nu(t)=(1,t,t^2,t^3)^T.$$
If we now consider $X\in\bbR^4$ corresponding to $\Psi\in\bgs^3\bbR^2$ of type $II$, i.e. of the form $\Psi^{ABC}=\xi^{(A}\xi^B\tau^{C)}$, we will find that the matrix of the endomorphism $L(X)$ reads    
$$\Big(L^A{}_H(X)\Big)=\tfrac29 (\xi^1\tau^2-\xi^2\tau^1)^2\bma\xi^1\xi^2&-(\xi^1){}^2\\(\xi^2){}^2&-\xi^1\xi^2\ema.$$
Although this matrix has nonzero entries, which means that the vector $X$ corresponding to $\Psi^{ABC}=\xi^{(A}\xi^B\tau^{C)}$ is \emph{not} null in either of the metrics $g^1$, $g^2$ or $g^3$, its determinant $\det(L(X))$ clearly vanishes. Thus
$$\ten(X,X,X,X)=\det(L(X))=0.$$
This shows that directions ${\rm dir}(X)$ corresponding to vectors $X$  having its spinorial image $\Psi$ of type $II$, are null with respect to the structural tensor $\ten$.

For vectors $X^i$ corresponding to $\Psi^{ABC}=\xi^{(A}\xi^B\eta^{C)}$ we have:
$$(X^1,X^2,X^3,X^4)=\Big((\xi^1)^2\eta^1,\tfrac13\xi^1(\xi^1\eta^2+2\xi^2\eta^1),\tfrac13\xi^2(2\xi^1\eta^2+\xi^2\eta^1),(\xi^2)^2\eta^2\Big).$$
Parametrizing $\xi$ and $\eta$ as:
$$(\xi^1,\xi^2)=(u,ut),\quad{\rm and}\quad(\eta^1,\eta^2)=(v,v(t+3s)),$$
we get
$$(X^1,X^2,X^3,X^4)=\Big(u^2v,u^2v(t+s),u^2v(t^2+2st),u^2v(t^3+3st^2)\Big).$$
This means that
\be {\rm dir}(X)=(1,t,t^2,t^3)^T+s(0,1,2t,3t^2)^T=\nu(t)+s\dot{\nu}{}(t),\label{g2man}\ee
i.e. that the directions of vectors $X$ corresponding to spinors $\Psi\in\bgs^3\bbR^2$ of type $II$ lie on the \emph{tangent variety to the twisted cubic} $\nu(t)$. 

In this way we see that the set of all null directions for the structural tensor $\ten$ \emph{stratifies} into directions lying on the twisted cubic $\nu$ (these are directions corresponding to $\Psi$s of type $N$, which are in addition null directions for the module $\Span_\bbR(g^1,g^2,g^3)$), and into more general directions lying strictly on the tangent variety to the twisted cubic (these are directions corresponding to $\Psi$s of strict type $II$; they are null with respect to $\ten$, and are \emph{not} null with respect to any of $g^1$, $g^2$ and $g^3$).

We have the following proposition.
    \begin{proposition}~\\
1) The set $\mathcal T$ of all directions in $\bbR^4$ that are null with respect to all three bilinear forms $g^1$, $g^2$ and $g^3$ is a twisted cubic. \\
2) The set $TV({\mathcal T})$ of all directions in $\bbR^4$ that are null with respect to the structural tensor $\ten$ is the tangent variety to the twisted cubic $\mathcal T$. 
      \end{proposition}
    \subsection{$G_2$ mode flying rule}\label{g2section}
    We are now in a position to teach a pilot of a flying saucer to perform a $G_2$ \emph{mode aerobatics man{\oe}uvre}. For this we need to equip flying saucer's configuration space with a $G_2$ \emph{contact structure}. This will be done by equipping a contact distribution of a flying saucer with the irreducible $\glg(2,\bbR)$ structure in dimension four compatible with the symplectic form on the distribution. 

Any flying saucer is equipped with a \emph{contact structure} with a contact form $\omega^0$. In coordinates, we have
    \be \omega^0=\der z-a\der x-b\der y,\label{cong21}\ee
    so that the symplectic form is 
    $$\der\omega^0=\der x\dz\der a+\der y\dz\der b.$$
    Comparing this with (\ref{2form}), we introduce 1-forms
    \be \omega^1=\der x,\quad\omega^2=\der y,\quad\omega^3=-\tfrac13\der b,\quad\omega^4=\der a,\label{cong22}\ee
    to achieve
    $$\der\omega^0=\omega^1\dz\omega^4-3\omega^2\dz\omega^3.$$
    The contact distribution $\mathcal D$ is spanned by the duals:
    \be Z_1=\partial_x+a\partial_z,\quad Z_2=\partial_y+b\partial_z,\quad Z_3=-3\partial_b,\quad Z_4=\partial_a.\label{cong23}\ee
    We use forms $(\omega^1,\omega^2,\omega^3,\omega^4)$ to define a module of symmetric bilinear forms $[g^1,g^2,g^3]$ as well as the structural tensor $\ten$ on the contact distribution by:
    $$g^1=\omega^1\omega^3-(\omega^2)^2,\quad g^2=(\omega^3)^2-\omega^2\omega^4,\quad g^3=\omega^1\omega^4-\omega^2\omega^3,$$
    and
    \be\ten=3(\omega^2)^2(\omega^3)^2-4\omega^1(\omega^3)^3-4(\omega^2)^3\omega^4+6\omega^1\omega^2\omega^3\omega^4-(\omega^1)^2(\omega^4)^2.\label{cong24}\ee
    In the above formulas, the multiplication of 1-forms $\omega^i$ means the \emph{symmetric tensor product}, e.g. $3\omega^2\omega^3=\tfrac32(\omega^2\otimes\omega^3+\omega^3\otimes\omega^2)$.
    According to the discussion in Section \ref{sec:gl2r} the tensor $\ten$ defined above reduces the structure group of the contact distribution $\mathcal D$ from $\glg(4,\bbR)$ to the irreducible $\glg(2,\bbR)$. Moreover, this reduction is done through $\cspg(\omega)$, with the symplectic form $\omega$ agreeing with the symplectic form $\der\omega^0$ coming from the contact form $\omega^0$ of the flying saucer. 

    A $G_2$ \emph{mode man{\oe}uvre} is a movement of the flying saucer such that at every moment of time its full velocity $\dot{\gamma}$ in the configuration space $\mathcal C$ is along a (tangent to the contact distribution) null direction for the above defined structural tensor $\ten$. There are two categories of this man{\oe}uvre: an \emph{easier} one, when $\dot{\gamma}$ lies on the tangent variety of the twisted cubic, and a \emph{more advanced} one, when $\dot{\gamma}$ lies on the twisted cubic. In our setting, the easier man{\oe}uvre means that saucer's velocity $\dot{\gamma}{}(t)$ at time $t$ must be \emph{parallel} to (compare with (\ref{g2man})):
    \be Z_1+\Big(T(t)+S(t)\Big)Z_2+\Big(T^2(t)+2S(t)T(t)\Big)Z_3+\Big(T^3(t)+3S(t)T^2(t)\Big)Z_4,\label{cong25}\ee
     where $T(t)$ and $S(t)$ are arbitrary functions of time (controls) that a pilot can adjust as he wishes. The more advanced man{\oe}uvre means that saucer's velocity $\dot{\gamma}{}(t)$ must be \emph{parallel} to 
     \be Z_1+T(t)Z_2+T^2(t)Z_3+T^3(t)Z_4,\label{cong26}\ee
     i.e. it must be such that it is null with respect to all three bilinear forms $g^1$, $g^2$ and $g^3$, in addition of being null with respect to $\ten$.
     \subsection{$G_2$ contact structures in dimension 5}
     We have to comment on the name of the man{\oe}uvre defined in the previous section. Why it is a $G_2$ man{\oe}uvre?

     For this we need a general definition of a $G_2$ contact structure. Here it is:
     \begin{definition} Let $M$ be a 5-dimensional contact manifold $M$, with a direction of a contact form $\omega^0$ and a contact distribution ${\mathcal D}=(\omega^0)^\perp$. A $G_2$ contact structure on $M$ is the reduction of the structure group of $\mathcal D$ from $\glg(4,\bbR)$ to the irreducible $\glg(2,\bbR)\subset\cspg(\der\omega^0)$.  
     \end{definition}
     This definition means that the contact manifold $M$ is a $G_2$ contact manifold if and only if, equivalently, it is equipped with a contact form $\omega^0$ and a symmetric forth rank tensor $\ten$, both given up to scale, on the contact distribution $\mathcal D$, and such that the common stabilizer of $\ten$ and $\der\omega^0$ in $\glg(4,\bbR)$,
     $$\begin{aligned}
       S_{\{\ten,\der\omega^0\}}=&\{A\in\glg(4,\bbR)\,\,{\rm s.t.}\\&\ten(AX,AX,AX,AX)=f(A)\ten(X,X,X,X),\\&\quad\quad\quad\quad\quad\quad\quad\quad\&\\&\der\omega^0(AX,AY)=h(A)\der\omega^0(X,Y),\,\, X,Y\in{\mathcal D}\},\end{aligned}$$
 is the irreducible $\glg(2,\bbR)$,
 $$S_{\{\ten,\der\omega^0\}}={\rm irreducible}\big(\glg(2,\bbR)\big)\subset\cspg(\der\omega^0)\subset\glg(4,\bbR).$$

 We say that two $G_2$ contact structures, represented by respective pairs $(\omega^0,\ten)$ and $(\bar{\omega}{}^0,\bar{\ten})$ on respective manifolds $M$ and $\bar{M}$  are locally equivalent, iff there exists a local diffeomorphism
 $\phi: M\to\bar{M}$ such that
 $$\phi^*(\bar{\omega}{}^0)=f_1\omega^0,\quad{\rm and}\quad\phi^*(\bar{\ten})=f_2\ten\quad{\rm mod}\,\,\omega^0,$$
 with some nonvanishing functions $f_1$ and $f_2$ on $M$. Local self-equivalences are called \emph{symmetries}. They form a \emph{group of local symmetries} of a $G_2$ contact structure. \emph{Infinitesimal symmetries} are vector fields $X$ on $M$ such that
 \be
 \big( {\mathcal L}_X\omega^0 \big)\dz\omega^0=0,\quad{\rm and}\quad {\mathcal L}_X\ten=h\ten+\omega^0\chi,\label{sychi}
 \ee
 where $h$ is a function on $M$, and $\chi$ is a symmetric rank 3 tensor on $M$. Infinitesimal symmetries form a Lie algebra of symmetries of a $G_2$ contact structure.

 It follows that $G_2$ contact structures in dimension 5 may have no infinitesimal symmetries at all, and that there is a precisely one (modulo local equivalence) $G_2$ contact structure with a maximal Lie algebra of symmetries. The name of a $G_2$ contact structure reflects the fact that the maximal Lie algebra of symmetries is isomorphic to the (split real form of the) exceptional simple Lie algebra $\mathfrak{g}_2$ of dimension 14.

 We can now state a theorem about the $G_2$ contact structure defining the $G_2$ mode man{\oe}uvres as described by (\ref{cong25}) and (\ref{cong26}). This structure has the contact form $\omega^0$ as in (\ref{cong21}) and the structural tensor $\ten$ as in (\ref{cong24}). We have the following theorem.

 \begin{proposition}
   The symmetry algebra of the $G_2$ contact structure on ${\mathcal C}_>0$ with the contact form $\omega^0=\der z - a \der x-b \der y$ and the structural tensor $\ten$ given by (\ref{cong24}) and (\ref{cong22}) is the split real form of the simple exceptional Lie algebra $\mathfrak{g}_2$. It is spanned by the following 14 vector fields on ${\mathcal C}_>$:
$$\begin{aligned}
X_1&=(y^3+xz)\partial_x+(yz-\tfrac19b^2x-\tfrac23by^2)\partial_y+(z^2-\tfrac{2}{27}b^3x-\tfrac13b^2y^2)\partial_z+\\&(az-a^2x-aby+\tfrac{1}{27}b^3)\partial_a+(bz-abx-3ay^2-\tfrac13b^2y)\partial_b\\
X_2&=x^2\partial_x+xy\partial_y+(xz-y^3)\partial_z+(z-ax-by)\partial_a-3y^2\partial_b\\
X_3&=-\tfrac12z\partial_x+\tfrac{1}{18}b^2\partial_y+\tfrac{1}{27}b^3\partial_z+\tfrac12a^2\partial_a+\tfrac12ab\partial_b\\
X_4&=-3y^2\partial_x+(\tfrac43by-z)\partial_y+\tfrac23b^2y\partial_z+ab\partial_a+(6ay+\tfrac13b^2)\partial_b\\
X_5&=\tfrac13y\partial_y+z\partial_z+a\partial_a+\tfrac23b\partial_b\\
X_6&=\tfrac92xy\partial_x+(\tfrac32y^2-bx)\partial_y+\tfrac12(9yz-b^2x)\partial_z+\tfrac12b^2\partial_a+\tfrac12(9z+3by-9ax)\partial_b\\
X_7&=-x\partial_y+3y^2\partial_z+b\partial_a+6y\partial_b\\
X_8&=x\partial_x+\tfrac23y\partial_y+z\partial_z+\tfrac13b\partial_b\\
X_{9}&=y\partial_x-\tfrac29b\partial_y-\tfrac19b^2\partial_z-a\partial_b\\
X_{10}&=x\partial_z+\partial_a\\
X_{11}&=y\partial_z+\partial_b\\
X_{12}&=\partial_x\\
X_{13}&=\partial_y\\
X_{14}&=\partial_z.
\end{aligned}
$$ 
\end{proposition} 
\begin{proof}
As before we solved the infinitesimal symmetry equations (\ref{sychi}) for $X$ with 
$\omega^0=\der z - a \der x-b \der y$ and $\ten$ as in  (\ref{cong24}) and (\ref{cong22}). We found a 14-dimensional space of solutions with the basis $(X_1, X_2,\dots, X_{14})$ as in the statement of the proposition.
\end{proof}
\subsection{Rolling balls as a joystick for $G_2$ aerobatics}
In this section we will use a double fibration
\begin{align}\label{df}
   \xymatrix{
        &\mathrm{G}_2/P_{1,2}  \ar[dl]_{\pi_2} \ar[dr]^{\pi_1} & \\
                   \mathrm{G}_2/P_2 &            & \mathrm{G}_2/P_1 }.
\end{align}
to show how a pilot can achieve the $G_2$ aerobatics man{\oe}uvres using a device - a joystick - consisting of two balls of radii ratio $3:1$ rolling on each other without slipping or twisting.  

In \eqref{df} $P_1$ and $P_2$ are respective two nonisomorphic 9-dimensional parabolic subgroups of the split real form of the exceptional group $G_2$, and $P_{1,2}=P_1\cap P_2$ is the 8-dimensional Borel subgroup of $G_2$. The homogeneous space $G_2/P_2$ has a natural $G_2$ invariant contact structure, and the space $G_2/P_1$ has a natural $G_2$ invariant $(2,3,5)$ distribution structure \cite{Br2}.

To describe these two structures, following \cite{katja}, we introduce two sets of coordinates on the 6-dimensional correspondence space $G_2/P_{1,2}$ adapted to the fibration \eqref{df}.

We start with the basis $(\theta^A)$ of the Maurer-Cartan forms on $G_2$ as in \cite{pawel}. They satisfy the following EDS:
 \begin{equation}\begin{aligned}\label{MaurerCartan}
&\der\theta^0=-6\theta^0\dz\theta^5+\theta^1\dz\theta^4-3\theta^2\dz\theta^3\\
&\\
&\der\theta^1=6\theta^0\dz\theta^9-3\theta^1\dz\theta^5-3\theta^1\dz\theta^8+3\theta^2\dz\theta^7\\
&\der\theta^2=2\theta^0\dz\theta^{10}+\theta^1\dz\theta^6-3\theta^2\dz\theta^5-\theta^2\dz\theta^8+2\theta^3\dz\theta^7\\
&\der\theta^3=2\theta^0\dz\theta^{11}+2\theta^2\dz\theta^6-3\theta^3\dz\theta^5+\theta^3\dz\theta^8+\theta^4\dz\theta^7\\
&\der\theta^4=6\theta^0\dz\theta^{12}+3\theta^3\dz\theta^6-3\theta^4\dz\theta^5+3\theta^4\dz\theta^8\\
&\\
&\der\theta^5=2\theta^0\dz\theta^{13}-\theta^1\dz\theta^{12}+\theta^2\dz\theta^{11}-\theta^3\dz\theta^{10}+\theta^4\dz\theta^9\\
&\der\theta^6=6\theta^2\dz\theta^{12}-4\theta^3\dz\theta^{11}+2\theta^4\dz\theta^{10}+2\theta^6\dz\theta^8\\
&\der\theta^7=-2\theta^1\dz\theta^{11}+4\theta^2\dz\theta^{10}-6\theta^3\dz\theta^9-2\theta^7\dz\theta^8\\
&\der\theta^8=-3\theta^1\dz\theta^{12}+\theta^2\dz\theta^{11}+\theta^3\dz\theta^{10}-3\theta^4\dz\theta^9-\theta^6\dz\theta^7\\
&\\
&\der\theta^9=-\theta^1\dz\theta^{13}-3\theta^5\dz\theta^9-\theta^7\dz\theta^{10}+3\theta^8\dz\theta^9\\
&\der\theta^{10}=-3\theta^2\dz\theta^{13}-3\theta^5\dz\theta^{10}-3\theta^6\dz\theta^9-2\theta^7\dz\theta^{11}+\theta^8\dz\theta^{10}\\
&\der\theta^{11}=-3\theta^3\dz\theta^{13}-3\theta^5\dz\theta^{11}-2\theta^6\dz\theta^{10}-3\theta^7\dz\theta^{12}-\theta^8\dz\theta^{11}\\
&\der\theta^{12}=-\theta^4\dz\theta^{13}-3\theta^5\dz\theta^{12}-\theta^6\dz\theta^{11}-3\theta^8\dz\theta^{12}\\
&\\
&\der\theta^{13}=-6\theta^5\dz\theta^{13}-6\theta^9\dz\theta^{12}+2\theta^{10}\dz\theta^{11}.
 \end{aligned}\end{equation}
 From this one sees that the 6-dimensional distribution
 $$D_6=\ker(\theta^5,\theta^6,\theta^8,\theta^9,\theta^{10},\theta^{11},\theta^{12},\theta^{13})$$
 on $G_2$ is integrable. Thus, $G_2$ is foliated by a $6$-dimensional manifolds. Let us concentrate on one leaf $M_6$ of this foliation. We have inclusion
 $$\iota: M_6\hookrightarrow G_2,$$
 and denoting by $\omega^A=\iota^*\theta^A$, we have
 $$\omega^5=\omega^6=\omega^8=\omega^9=\omega^{10}=\omega^{11}=\omega^{12}=\omega^{13}=0.$$
 Thus on the leaf $M_6$, which we now identify with $M_6=G_2/P_{1,2}$, we have the following EDS satisfied by a coframe $(\omega^0,\omega^1,\omega^2,\omega^3,\omega^4,\omega^7)$ on $M_6$:
 \be\begin{aligned}
 &\der\omega^0=\omega^1\dz\omega^4-3\omega^2\dz\omega^3,\\&\der\omega^1=3\omega^2\wedge\omega^7,\\&\der\omega^2=2\omega^3\wedge\omega^7,\\& \der\omega^3=\omega^4\wedge\omega^7,\\&\der\omega^4=0,\\& \der\omega^7=0.\end{aligned}\label{syka}
\ee
Integrating this system, starting from $\omega^7=-\der x^5$, and $\omega^4=\der x^4$, gives first $\der(\theta^3-x^5\der x^4)=0$, i.e. $\theta^3=\der x^3+x^5\der x^4$, then $\der(\theta^2-2x^5\der x^3-(x^5)^2\der x^4)=0$, i.e. $\theta^2=\der x^2+2x^5\der x^3+(x^5)^2\der x^4$, and by continuing, finally:
 \begin{equation}\label{omegas}
\begin{aligned}
& \omega^0=\der x^0+x^1\der x^4- 3 x^2\der x^3\\
 &\omega^1=\der x^1+3 x^5\der x^2+ 3 (x^5)^2 \der x^3+(x^5)^3\der x^4\\
 &  \omega^2=\der x^2 +2 {x^5} \der x^3 +(x^5)^2 \der x^4\\
  &  \omega^3=\der x^3+ x^5 \der x^4\\
  &  \omega^4=\der x^4,\\
  &\omega^7=-\der x^5.
 \end{aligned}
 \end{equation}
 This gives a coordinate system $(x^0,x^1,x^2,x^3,x^4,x^5)$ on $M_6=G_2/P_{1,2}$, respecting the fibration
\begin{align}\label{df2}
   \xymatrix{
        &\mathrm{G}_2/P_{1,2}  \ar[dl]_{\pi_2}^{e_7=-\partial_{x^5}}& \\
                   \mathrm{G}_2/P_2 &            & .}
\end{align}
In theses coordinates the projection $\pi_2$ forgets about the last component of $(x^0,x^1,x^2,x^3,x^4,x^5)$, i.e. $$\pi_2 : (x^0,x^1,x^2,x^3,x^4,x^5)\mapsto (x^0,x^1,x^2,x^3,x^4).$$
Slight change in the order of integration in \eqref{syka}, i.e. a start with $\omega^7=-\der y^4$, $\omega^4=\der y^5$, gives another local coordinate system  $(y^0,y^1,y^2,y^3,y^4,y^5)$ on $M_6=\mathrm{G}_2/P_{1,2}$ in which the basis of 1-forms is
\begin{equation}\label{newomegas}
\begin{aligned}
& \omega^0=\der y^0-y^5\der y^1- 3 y^4 y^5\der y^2-3(y^2+y^5(y^4)^2)\der y^3\\
 &\omega^1=\der y^1+3 y^4\der y^2+ 3 (y^4)^2 \der y^3\\
 &  \omega^2=\der y^2 +2 {y^4} \der y^3 \\
  &  \omega^3=\der y^3- y^5 \der y^4\\
  &  \omega^4=\der y^5,\\
  &\omega^7=-\der y^4.
 \end{aligned}
\end{equation}
This respects the fibration
\begin{align}\label{df1}
   \xymatrix{
        &\mathrm{G}_2/P_{1,2}  \ar[dr]^{\pi_1}_{e_4=\partial_{y^5}} \\
                  & & \mathrm{G}_2/P_1 \,\,,}
\end{align}
in which the projection $\pi_1$ forgets about the last component in $(y^0,y^1,y^2,y^3,y^4,y^5)$, so that we have $$\pi_1 : (y^0,y^1,y^2,y^3,y^4,y^5)\mapsto (y^0,y^1,y^2,y^3,y^4).$$
A change of coordinates from $(x^0,x^1,x^2,x^3,x^4, x^5)$ to $(y^0,y^1,y^2,y^3, y^4, y^5)$ is given by
\be
\begin{aligned}\label{coordchange}
  &y^0=x^0+x^1x^4+3x^5 x^2x^4-(x^5)^3(x^4)^2,\\
  & y^1=x^1+(x^5)^3x^4,\\
  & y^2=x^2-(x^5)^2x^4,\\
  & y^3=x^3+x^5x^4,\\
  & y^4=x^5, \\
  & y^5=x^4.
\end{aligned}
\ee
We can thus coordinate the three manifolds $G_2/P_{1,2}$, $G_2/P_2$ and $G_2/P_1$ as in the following diagram:
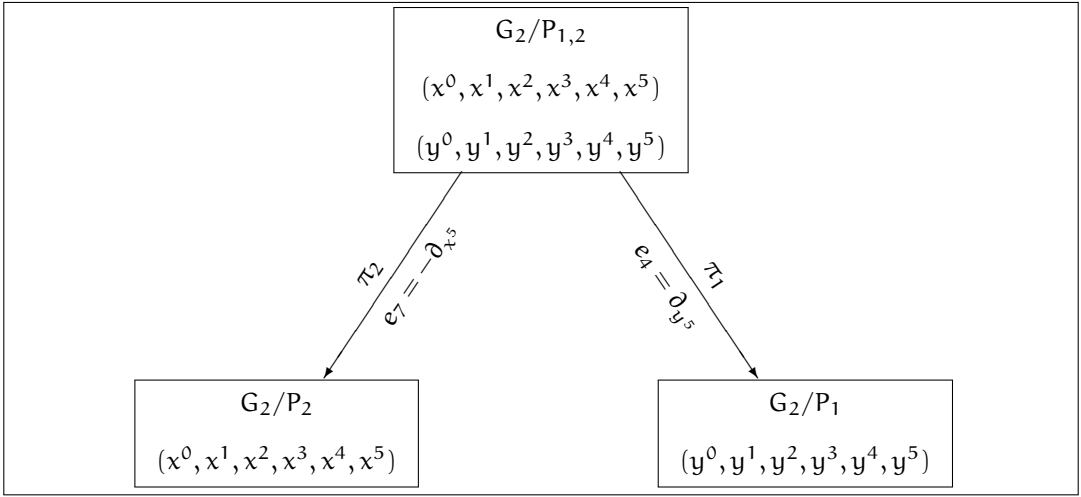
\begin{center}
\framebox{\begin{picture}(400,180)
\put(200,150){\makebox(0,0){\framebox{$\begin{array}{c}
G_2/P_{1,2}\\[10pt]
(x^0,x^1,x^2,x^3,x^4,x^5)\\[10pt]
(y^0,y^1,y^2,y^3,y^4,y^5)
\end{array}$}}}
\put(100,20){\makebox(0,0){\framebox{$\begin{array}{c}
G_2/P_2\\[10pt]
(x^0,x^1,x^2,x^3,x^4,x^5)
\end{array}$}}}
\put(300,20){\makebox(0,0){\framebox{$\begin{array}{c}
G_2/P_1\\[10pt]
(y^0,y^1,y^2,y^3,y^4,y^5)
\end{array}$}}}
\put(170,119){\vector(-2,-3){52}}
\put(230,119){\vector(2,-3){52}}
\put(130,75){\rotatebox{56}{$\pi_2$}}
\put(138,60){\rotatebox{56}{$e_7=-\partial_{x^5}$}}
\put(260,83){\rotatebox{-56}{$\pi_1$}}
\put(232,90){\rotatebox{-56}{$e_4=\partial_{y^5}$}}
\end{picture}}
\end{center}

There is a natural contact structure on $G_2/P_2$ given by a contact distribution ${\mathcal D}_4$ spanned by
\be Z_1=\partial_{x^4}-x^1\partial_{x^0},\,\quad\,Z_2=\partial_{x^3}+3x^2\partial_{x^0},\,\quad\,Z_3=\partial_{x^2},\,\quad\,Z_4=\partial_{x^1}.\label{z1234}\ee
This defines a \emph{flat} $G_2$ contact structure on $G_2/P_2$ with the structural tensor $\Upsilon$ on ${\mathcal D}_4$ defined in terms of \eqref{omegas} as in \eqref{cong24}. Note that, although $(\omega^i)$ in \eqref{omegas} depend on the 6th variable $x^5$, the tensor $\Upsilon$ defined by \eqref{cong24} does not depend on it,
$$\Upsilon=3(\der x^2)^2(\der x^3)^2-4\der x^1(\der x^3)^3-4(\der x^2)^3\der x^4+6\der x^1\der x^2\der x^3\der x^4-(\der x^1)^2(\der x^4)^2.$$
Let us now introduce a basis $(e_0,e_1,e_2,e_3,e_4,e_7)$ of vector fields on $G_2/P_{1,2}$ dual to the 1-forms $(\omega^0,\omega^1,\omega^2,$ $\omega^3,\omega^4,\omega^7)$. Then the fibers of the fibration $G_2/P_{1,2}\to G_2/P_2$ are tangent to $e_7$ and fibers of the fibration $G_2/P_{1,2}\to G_2/P_1$ are tangent to $e_4$.
Looking at the system \eqref{syka} we find that the only nonzero commutator of vector fields $(e_3,e_7)$ with the fiber generator $e_4$ of  $G_2/P_{1,2}\to G_2/P_1$ is $[e_4,e_7]=-e_3$. Thus, rank 2 distribution $\tilde{\mathcal D}{}_2$ spanned on $G_2/P_{1,2}$ by vector fields $e_3$ and $e_7$ \emph{descends} to a well defined rank 2-distribution
$${\mathcal D}_2=\Span(\pi_{1*}e_7,\pi_{1*}e_3)$$ on $G_2/P_1$. Since on $G_2/P_{1,2}$ we have
$$[e_7,e_3]=2e_2, \quad [e_7,e_2]=-3 e_1,\quad [e_3,e_2]=3 e_0,$$
and
$$[e_4,e_2]=0, \quad [e_4,e_0]=0,\quad [e_4,e_1]=-e_0,$$
then the filtration 
$$\Span(e_7,e_3)\subset\Span(e_7,e_3,e_2)\subset\Span(e_7,e_3,e_2,e_1,e_0)$$
of distributions of respective rank 2,3,5 on $G_2/P_{1,2}$ \emph{descends} to a $G_2$ invariant filtration 
$${\mathcal D}_2\subset[{\mathcal D}_2,{\mathcal D}_2]\subset[{\mathcal D}_2,[{\mathcal D}_2,{\mathcal D}_2]]={\mathrm T}(G_2/P_1)$$
of distributions of respective rank 2,3,5 on $G_2/P_1$.

Thus we see that on $G_2/P_2$ we have a $G_2$ invariant $G_2$ \emph{contact structure} $({\mathcal D}_4,\Upsilon)$ and on $G_2/P_1$ we have a $G_2$ invariant $(2,3,5)$ distribution structure $({\mathcal D}_2,[{\mathcal D}_2,{\mathcal D}_2],[{\mathcal D}_2,[{\mathcal D}_2,{\mathcal D}_2]])$.

The explicit forms of the vector fields $e_3$ and $e_7$ on $G_2/P_{1,2}$ in coordinates $(y^0,y^1,y^2,y^3,y^4,y^5)$ are
$$e_3=\partial_{y^3}-2y^4\partial_{y^2}+3(y^4)^2\partial_{y^1}+3y^2\partial_{y^0},\quad e_7=-\partial_{y^4}-y^5e_3.$$

This, in particular, enables us to interpret $G_2/P_{1,2}$ as the bundle over $G_2/P_1$ of \emph{directions} in the 2-distribution ${\mathcal D}_2$: the fiber coordinate $y^5$ in the fibration $G_2/P_{1,2}\to G_2,P_1$ corresponds to a choice of a direction $Y(y^5)={\mathrm dir}\Big(-\partial_{y^4}-y^5 (\pi_{1*}e_3)\Big)$ in the distribution ${\mathcal D}_2=\Span(\partial_{y^4},\pi_{1*}e_3)$.

This enables us to \emph{lift any curve} $\gamma(t)$  \emph{tangent to the distribution} ${\mathcal D}_2$ from $G_2/P_1$ to $G_2/P_{1,2}$. Indeed, a curve
\be \gamma_y(t)=\Big(y^0(t),y^1(t),y^2(t),y^3(t),y^4(t)\Big),\label{gat}\ee
as tangent to ${\mathcal D}_2$ at every instant of time $t$, defines a \emph{direction}
$${\mathrm dir}(\dot{\gamma}_y(t))=\Big(\dot{y}{}^0(t),\dot{y}{}^1(t),\dot{y}{}^2(t),\dot{y}{}^3(t),\dot{y}{}^4(t)\Big)$$
in the distribution ${\mathcal D}_2$, i.e. a \emph{point} in the fiber in $G_2/P_{1,2}$ over the point $\gamma_y(t)\in G_2/P_1$. Changing $t$ along $\gamma_y(t)$ we collect these points in $G_2/P_{1,2}$ obtaining a curve $\tilde{\gamma}(t)\subset G_2/P_{1,2}$ above $\gamma_y(t)$. This is the lift of $\gamma_y(t)$. Explicitly, any curve $\gamma_y(t)=\Big(y^0(t),y^1(t),y^2(t),y^3(t),y^4(t)\Big)\subset G_2/P_1$ has a tangent vector
$$
\begin{aligned}
  \dot{\gamma}_y(t)=&\dot{y}{}^4\partial_{y^4}+\dot{y}{}^3\partial_{y^3}+\dot{y}{}^2\partial_{y^2}+\dot{y}{}^1\partial_{y^1}+\dot{y}{}^0\partial_{y^0}\\
=&\dot{y}{}^4\partial_{y^4}+\dot{y}{}^3\pi_{1*}e_3+(\dot{y}{}^2+2y^4\dot{y}{}^3)\partial_{y^2}+(\dot{y}{}^1-3(y^4)^2\dot{y}{}^3)\partial_{y^1}+(\dot{y}{}^0-3y^2\dot{y}{}^3)\partial_{y^0},
\end{aligned}
$$
and for this to be tangent to ${\mathcal D}_2$ we need to have:
\be \dot{y}{}^0=3y^2\dot{y}{}^3\quad\&\quad \dot{y}{}^1=3(y^4)^2\dot{y}{}^3\quad\&\quad \dot{y}{}^2=-2y^4\dot{y}{}^3.\label{relan}\ee
If these tangency to ${\mathcal D}_2$ relations are satisfied, a tangent vector to $\gamma_y(t)$ at $t$ is
$\dot{\gamma}{}_y(t)=\dot{y}{}^4\partial_{y^4}+\dot{y}{}^3\pi_{1*}e_3$.  This  defines a direction
$${\mathrm dir}(\dot{\gamma}{}_y)(t)=\partial_{y^4}+\Big(\frac{\dot{y}{}^3}{\dot{y}{}^4}\Big)\,(\pi_{1*}e_3)$$
in ${\mathcal D}_2$, and in turn the lift
$$\begin{aligned}
  \tilde{\gamma}(t)=&\Big(y^0(t),y^1(t),y^2(t),y^3(t),y^4(t),y^5(t)\Big)\\
    =&\Big(y^0(t),y^1(t),y^2(t),y^3(t),y^4(t),\frac{\dot{y}{}^3(t)}{\dot{y}{}^4(t)}\Big)\end{aligned}$$
of the curve $\gamma_y(t)$ from $G_2/P_1$ to $G_2/P_{1,2}$.
It is now convenient to rewrite this curve on $G_2/P_{1,2}$ in coordinates $(x^0,x^1,x^2,x^3,x^4,x^5)$ adapted to the fibration $G_2/P_2$. According to \eqref{coordchange} we have
$$\tilde{\gamma}(t)=\Big(x^0(t),x^1(t),x^2(t),x^3(t),x^4(t),x^5(t)\Big)$$ with 
\be\begin{aligned}
  x^0(t)=&\big(y^0-y^1y^5-2y^2y^4y^5-(y^4)^3(y^5)^2\big)(t),\\
  x^1(t)=&\big(y^1-(y^4)^3 y^5\big)(t)\\
  x^2(t)=&\big(y^2+(y^4)^2 y^5\big)(t)\\
  x^3(t)=&\big(y^3-y^4 y^5\big)(t)\\
  x^4(t)=&y^5(t)\\
  x^5(t)=&y^4(t)\\
\end{aligned}
\label{xxx}\ee
and
$$y^5(t)=\frac{\dot{y}{}^3(t)}{\dot{y}{}^4(t)}.$$
Projecting this by using $\pi_2$ gives a curve
$$\begin{aligned}
  \gamma_x(t)&=\pi_2(\tilde{\gamma}(t))=\Big(x^0,x^1,x^2,x^3,x^4,x^5\Big)(t)\\
  &=\Big(y^0-y^1\frac{\dot{y}{}^3}{\dot{y}{}^4}-2y^2y^4\frac{\dot{y}{}^3}{\dot{y}{}^4}-(y^4)^3(\frac{\dot{y}{}^3}{\dot{y}{}^4})^2,y^1-(y^4)^3 \frac{\dot{y}{}^3}{\dot{y}{}^4},y^2+(y^4)^2 \frac{\dot{y}{}^3}{\dot{y}{}^4},y^3-y^4 \frac{\dot{y}{}^3}{\dot{y}{}^4}\Big)(t).\end{aligned}$$

We thus uniquely associated a curve $\gamma_x(t)$ in $G_2/P_2$ to any curve $\gamma_y(t)$ which in $G_2/P_1$ is tangent to the $(2,3,5)$ distribution ${\mathcal D}_2$. This curve is \emph{very special} in $G_2/P_2$. Actually it is both \emph{tangent to the contact distribution} ${\mathcal D}_4$ in $G_2/P_2$ as well it is \emph{simultaneously null with respect to all three bilinear forms} $g^1$, $g^2$, and $g^3$ defining the $G_2$ contact structure on $G_2/P_2$. To see this we find a tangent vector to $\gamma_x(t)$ at every moment of time $t$. Differentiating $\gamma_x(t)$ with respect to time, and using relations \eqref{relan} and \eqref{xxx},  we get
$$\begin{aligned}
  \dot{\gamma}_x(t)=&\dot{x}{}^4\partial_{x^4}+\dot{x}{}^3\partial_{x^3}+\dot{x}{}^2\partial_{x^2}+\dot{x}{}^1\partial_{x^1}+\dot{x}{}^0\partial_{x^0}\\
  =&\frac{\der}{\der t}\Big(\frac{\dot{y}{}^3(t)}{\dot{y}{}^4(t)}\Big)\Big(-\big(y^1+3y^2y^4+2(y^4)^3\frac{\dot{y}{}^3(t)}{\dot{y}{}^4(t)}\big)\partial_{x^0}-(y^4)^3\partial_{x^1}+(y^4)^2\partial_{x^2}-y^4\partial_{x^3}+\partial_{x^4}\Big)\\
  =&\frac{\der}{\der t}\Big(\frac{\dot{y}{}^3(t)}{\dot{y}{}^4(t)}\Big)~\Big(Z_1-y^4 Z_2+(y^4)^2 Z_3-(y^4)^3Z_4\Big),
\end{aligned}
$$
where $Z_1$, $Z_2$, $Z_3$ and $Z_4$ are (restricted to the curve $\gamma_x(t)$) vector fields spanning the contact distribution ${\mathcal D}_4$, as in \eqref{z1234}. Thus, the curve $\gamma_x(t)$ is not only tangent to the contact distribution ${\mathcal D}_4$ in $G_2/P_2$, but also its tangent vectors always lie on a twisted cubic
$$Z_1+T(t)Z_2+T^2(t)Z_3+T^3(t)Z_4$$
parametrized in ${\mathcal D}_4$ by $T(t)=-y^4(t)$.    


In this way, taking any curve $\gamma_y(t)$ in $G_2/P_1$ tangent to the $(2,3,5)$ distribution ${\mathcal D}_2$, then lifting it to a curve $\tilde{\gamma}(t)$ in $G_2/P_{1,2}$, and finally projecting this one to $G_2/P_2$, we get a curve $\gamma_x(t)$ in $G_2/P_2$ with velocity vector on the tangent variety to the twisted cubic, thus performing an \emph{advanced} $G_2$ man{\oe}uvre.
\section{Outlook}
\subsection{Achievability}
A flying saucer performing an attacking, landing or $G_2$ man{\oe}uvre, is a control system. It moves in the configuration space $C_>$, with coordinate system $(x,y,z,a,b)$ in a way such that its velocity $\dot{\gamma}$ is
\begin{itemize}
\item[(A)] $\dot{\gamma}=u_1\big(3u_3Z_1+Z_3\big)+u_2\big(u_3Z_2+Z_4\big)$ when in attacking mode,
\item[(L)] $\dot{\gamma}=u_3\big((1+b^2)u_2+3ab u_1\big)Z_1-u_3(\big(abu_2+3(1+a^2)u_1\big)Z_2+u_1Z_3+u_2Z_4$ when in landing mode,
\item[($G_2$)]
  \begin{itemize}
  \item[(s)] $\dot{\gamma}=u_1Z_1+u_1(u_2+u_3)Z_2+u_1(u_2^2+2u_3u_2)Z_3+u_1(u_2^3+3u_3u_2^2)Z_4$ when in a simpler $G_2$ mode, and
    \item[(d)] $\dot{\gamma}=u_1Z_1+u_1u_2Z_2+u_1u_2^2Z_3+u_1u_2^3Z_4$ when in a more stringent $G_2$ mode.
    \end{itemize}
\end{itemize}
Here the functions $u_1,u_2,u_3$ are arbitrary functions of time. They represents \emph{controls} of the system. Vector fields $Z_1,Z_2,Z_3,Z_4$ are given by (\ref{cong23}). They span the contact distribution.

A natural question in control theory is if, starting from a given configuration, one can move the system using available controls to any other position in configuration space. In flying saucer context one can ask if, only using one of the man{\oe}uvres (A), (L), ($G_2$)(s) or ($G_2$)(d) one can drive a saucer from any point in the configuration space starting to any other one.

To answer this we invoke the Nagano-Sussman theorem.
\begin{theorem}
  Let $\mathcal F=\{Y_i\}$ be a family of vector fields $Y_i$ on a manifold $M$. Suppose that a finite number of brackets of $Y_i$s and finite number of iterations of these brackets generate ${\rm T}_qM$ at every $q\in M$ (we say that $\mathcal F$ is \emph{bracket generating}). Then the orbit of this family of vector fields at each point is $M$.     
\end{theorem}

We now apply this theorem to the flying saucer in the \emph{attacking mode}. We can chose particular controls to define a family of vector fields that satisfy the assumption of the Nagano-Sussman theorem. Indeed,
\begin{itemize}
\item taking $u_1=1,u_2=u_3=0$ defines a vector field $Y_1=Z_3$,
\item taking $u_2=1, u_1=u_3=0$ defines $Y_2=Z_4$,
\item taking $u_1=u_3=1, u_2=0$ defines $Y_3=3Z_1+Z_3$, and
\item taking $u_1=0, u_2=u_3=1$ defines $Y_4=Z_2+Z_4$. \end{itemize}
All four vector fields $Y_1,Y_2,Y_3,Y_4$ are tangent to the contact distribution ${\mathcal D}$ and null with respect to the metric, which on $\mathcal D$ defines the attacking man{\oe}uvre. Since $[Y_2,Y_3]=3\partial_z$ and $Y_1\dz Y_2\dz Y_3\dz Y_4\dz\partial_z\neq 0$ at each point, then the family $F=\{Y_1,Y_2,Y_3,Y_4\}$ of vector fields on ${\mathcal C}_>$ is bracket generating at each point of ${\mathcal C}_>$. Thus, according to Nagano-Sussman theorem we can achieve any point of ${\mathcal C}_>$ starting from any other point by going along integral curves of vector fields from family $\mathcal F$.

For the achievability in the landing man{\oe}uvre we take
\begin{itemize}
\item $Y_1=Z_3$ by puting  $u_1=1,u_2=u_3=0$,
\item $Y_2=Z_4$ by puting $u_2=1, u_1=u_3=0$,
\item $Y_3=Z_3-3(1+a^2)Z_2+3ab Z_1$ by puting $u_1=u_3=1, u_2=0$,
\item $Y_4=Z_4-ab Z_2+(1+b^2)Z_1$ by puting $u_2=u_3=1, u_1=0$.
\end{itemize}
Here we see that $[Y_1,[Y_2,[Y_2,Y_3]]]=9\partial_z$, and we check that $Y_1\dz Y_2\dz Y_3\dz Y_4\dz\partial_z\neq 0$ at each point. Thus, again we found a family of vector fields ${\mathcal F}=\{Y_1,Y_2,Y_3,Y_4\}$ tangent to $\mathcal D$ and null with respect to the metric defining the landing man{\oe}uvre, which is bracket generating at each point of ${\mathcal C}_>$. By Nagano-Sussman theorem going along the integral curves of vector fields from this family enables as to achieve any point of ${\mathcal C}_>$ from any other point.

Finally for the achivability in the $G_2$ mode man{\oe}uvres, it is enough to show that the more stringent man{\oe}uvres ($G_2$)(d) can be used to go everywhere from anywhere. Taking $u_1=1$, and $u_2$ in values $\{0,1,-1,2\}$ we obtain $Y_1=Z_1$, $Y_2=Z_1+Z_2+Z_3+Z_4$, $Y_3=Z_1-Z_2+Z_3-Z_4$, and $Y_4=Z_1+2Z_2+4Z_3+8Z_4$. We have $[Y_2,Y_1]=\partial_z$ and  $Y_1\dz Y_2\dz Y_3\dz Y_4\dz\partial_z\neq 0$ at each point. Thus we see that we can achieve any point of the configuration space from any other point by a successive ($G_2$)(d) man{\oe}uvres.

Summarizing we have:
\begin{proposition}
A flaying saucer performing a sequence of (A) man{\oe}uvres, or a sequence of (L) man{\oe}uvres, or a sequence of ($G_2$) man{\oe}uvres can reach any point of its configuration space starting from any other point.   
  \end{proposition}

\subsection{Arbitrarines in the choice of twisted cubic in the $G_2$ mode man{\oe}uvre}
There is a frustrating difference in our definitions of attacking and landing modes of a saucer, and the definition of the $G_2$ modes. The former definitions were expressed in a `human accessible terms', such as `keep your velocity parallel/orthogonal to the axis of rotation', whereas the latter needed a rather arbitrary placement of the twisted cubic on the contact distribution. Not only such terms as `contact distribution' or `twisted cubic' are \emph{not} `human accessible', but also we had a \emph{huge arbitrariness} in choosing a twisted cubic on the contact distribution $\mathcal D$. Actually, whenever we chose a twisted cubic in a way such that its corresponding module $\Span[g^1,g^2,g^3]$ (or its structural tensor $\ten$) is \emph{compatible}\footnote{in the sense that the common stabilizer of $\ten$ and $\der\omega^0$ is irreducible $\glg(2,\bbR)$} with the contact form $\der\omega^0$, we always define a $G_2$ contact structure on saucer's configuration space. However, in general this structure will not be \emph{flat}, i.e. it will not have $G_2$ symmetry. Nevertheless there are possibilities of choosing a twisted cubic on $\mathcal D$, apparently different from the one we have chosen in Section \ref{g2section}, which also defines a $G_2$ contact structure with $G_2$ symmetry on saucer's configuration space. Here is an example:

Take $\omega^0=\der z-a \der x-b\der y$ as always, and
$$\omega^1=\frac{y\der x-x\der y}{y},\quad\omega^2=\frac{\der x}{x},\quad\omega^3=-\tfrac13y(x\der a+y\der b),\quad\omega^4=-\frac{y^2\der b}{x}.$$
We claim that via (\ref{cong24}) this defines a structural tensor $\ten$ which is compatible with $\der\omega^0$, i.e. the stabilizer of $\ten$ and $\der\omega^0$ in $\glg({\mathcal D})$ is the irreducible $\glg(2,\bbR)$. Moreover, one can check that the $G_2$ contact structure defined by $\mathcal D$ and $\ten$ on ${\mathcal C}_>$ has 14 infinitesimal symmetries which form $\mathfrak{g}_2$. Thus one can use \emph{this} realization of the flat $G_2$ contact structure to define $G_2$ man{\oe}uvres of the flying saucer.

Something is missing here: we start with a flat $G_2$ contact structure and then, by identifying its contact distribution with the velocity space of the flying saucer we are able to define the man{\oe}uvre. The situation was very different in case of man{\oe}uvres (A) and (L). We defined them knowing nothing about the corresponding contact geometry. The geometry was miraculously defined by the man{\oe}uvre. We need such an approach for the $G_2$ story, but we are unable to find it!

\section{Acknowledgments}
The authors would like to thank Katja Sagerschnig and Travis Willse
for many helpful conversations. 
Special thanks are due to Jan Gutt for the idea of a $G_2$ joystick, 
which was suggested during a beer session with the second author 
at Jabeerwocky, one craft beer pubs in Warsaw.

\end{document}